\documentclass[english,ruled, 10pt]{article}
\usepackage[T1]{fontenc}
\usepackage[utf8]{inputenc}
\usepackage{amsfonts,amsmath, amssymb, amsthm}
\usepackage{geometry}
\usepackage{algorithm, algpseudocode}
\usepackage{nicematrix}
\usepackage{bm}
\usepackage{xcolor}
\usepackage{subcaption} 
\usepackage{diagbox}
\usepackage{adjustbox}
\usepackage{lipsum}
\usepackage{graphicx} 
\usepackage{booktabs}
\usepackage{optidef}
\usepackage{pstricks, pst-tree}
\usepackage[square,numbers]{natbib}
\usepackage{mathabx}
\usepackage{mathtools}
\usepackage{listings}
\usepackage{float}
\usepackage{caption}
\usepackage{pifont}
\usepackage{babel}
\usepackage{multirow}
\usepackage{hyperref}
\hypersetup{
	colorlinks=true,
	linkcolor=red,
	citecolor = blue
}

\usepackage{changes}
\newcommand{\stkout}[1]{\ifmmode\text{\sout{\ensuremath{#1}}}\else\sout{#1}\fi}
\setdeletedmarkup{\stkout{#1}}
\definechangesauthor[name={Digvijay Boob}, color=blue]{DB}

\theoremstyle{plain}
\newtheorem*{assumption*}{\protect\assumptionname}

\theoremstyle{remark}
\newtheorem*{remark*}{\protect\remarkname}

\theoremstyle{plain}
\newtheorem{remark}{\protect\remarkname}

\theoremstyle{plain}
\newtheorem{theorem}{\protect\theoremname}

\theoremstyle{definition}
\newtheorem{definition}{\protect\definitionname}

\theoremstyle{plain}
\newtheorem{assumption}{\protect\assumptionname}

\theoremstyle{plain}
\newtheorem{proposition}{\protect\propositionname}

\theoremstyle{plain}
\newtheorem{lemma}{\protect\lemmaname}

\theoremstyle{plain}
\newtheorem{corollary}{\protect\corollaryname}

\providecommand{\assumptionname}{Assumption}
\providecommand{\corollaryname}{Corollary}
\providecommand{\definitionname}{Definition}
\providecommand{\lemmaname}{Lemma}
\providecommand{\propositionname}{Proposition}
\providecommand{\remarkname}{Remark}
\providecommand{\theoremname}{Theorem}

\numberwithin{equation}{section}

\newenvironment{keywords}{
  \par\noindent
  \textbf{Keywords:}
}{
  \par
}

\newif\ifsiam

\siamfalse 

\title{Optimal Primal-Dual Algorithm with Last iterate Convergence Guarantees for Stochastic Convex Optimization Problems 

}
\author{
Digvijay Boob\thanks{dboob@smu.edu, Operations Research and Engineering Management, Southern Methodist University}
\hspace{6em}
Mohammad Khalafi\thanks{mohamadk@smu.edu, Operations Research and Engineering Management, Southern Methodist University}
}
\date{}

\begin{document}
\allowdisplaybreaks
\global\long\def\vertiii#1{\left\vert \kern-0.25ex  \left\vert \kern-0.25ex  \left\vert #1\right\vert \kern-0.25ex  \right\vert \kern-0.25ex  \right\vert }%
\global\long\def\matr#1{\bm{#1}}%
\global\long\def\til#1{\tilde{#1}}%
\global\long\def\wt#1{\widetilde{#1}}%
\global\long\def\wh#1{\widehat{#1}}%
\global\long\def\wb#1{\widebar{#1}}%
\global\long\def\mcal#1{\mathcal{#1}}%
\global\long\def\mbb#1{\mathbb{#1}}%
\global\long\def\mtt#1{\mathtt{#1}}%
\global\long\def\ttt#1{\texttt{#1}}%
\global\long\def\inner#1#2{\langle#1,#2\rangle}%
\global\long\def\inter#1{\text{int }#1}%
\global\long\def\rinter#1{\text{rint }#1}%
\global\long\def\binner#1#2{\big\langle#1,#2\big\rangle}%
\global\long\def\Binner#1#2{\Big\langle#1,#2\Big\rangle}%
\global\long\def\br#1{\left(#1\right)}%
\global\long\def\bignorm#1{\bigl\Vert#1\bigr\Vert}%
\global\long\def\Bignorm#1{\Bigl\Vert#1\Bigr\Vert}%
\global\long\def\setnorm#1{\Vert#1\Vert_{-}}%
\global\long\def\rmn#1#2{\mathbb{R}^{#1\times#2}}%
\global\long\def\deri#1#2{\frac{d#1}{d#2}}%
\global\long\def\pderi#1#2{\frac{\partial#1}{\partial#2}}%
\global\long\def\onebf{\mathbf{1}}%
\global\long\def\zero{\mathbf{0}}%

\global\long\def\norm#1{\lVert#1\rVert}%
\global\long\def\bnorm#1{\big\Vert#1\big\Vert}%
\global\long\def\Bnorm#1{\Big\Vert#1\Big\Vert}%

\global\long\def\brbra#1{\big(#1\big)}%
\global\long\def\Brbra#1{\Big(#1\Big)}%
\global\long\def\rbra#1{(#1)}%
\global\long\def\sbra#1{[#1]}%
\global\long\def\bsbra#1{\big[#1\big]}%
\global\long\def\Bsbra#1{\Big[#1\Big]}%
\global\long\def\cbra#1{\{#1\}}%
\global\long\def\bcbra#1{\big\{#1\big\}}%
\global\long\def\Bcbra#1{\Big\{#1\Big\}}%

\global\long\def\grad{\nabla}%
\global\long\def\Expe{\mathbb{E}}%
\global\long\def\rank{\text{rank}}%
\global\long\def\range{\text{range}}%
\global\long\def\diam{\text{diam}}%
\global\long\def\epi{\text{epi }}%
\global\long\def\inte{\operatornamewithlimits{int}}%
\global\long\def\cov{\text{Cov}}%
\global\long\def\argmin{\operatornamewithlimits{argmin}}%
\global\long\def\argmax{\operatornamewithlimits{argmax}}%
\global\long\def\tr{\operatornamewithlimits{tr}}%
\global\long\def\dis{\operatornamewithlimits{dist}}%
\global\long\def\sign{\operatornamewithlimits{sign}}%
\global\long\def\prob{\mathbb{P}}%
\global\long\def\st{\operatornamewithlimits{s.t.}}%
\global\long\def\dom{\text{dom}}%
\global\long\def\diag{\text{diag}}%
\global\long\def\and{\text{and}}%
\global\long\def\st{\text{s.t.}}%
\global\long\def\Var{\operatornamewithlimits{Var}}%
\global\long\def\raw{\rightarrow}%
\global\long\def\law{\leftarrow}%
\global\long\def\Raw{\Rightarrow}%
\global\long\def\Law{\Leftarrow}%
\global\long\def\vep{\varepsilon}%
\global\long\def\dom{\operatornamewithlimits{dom}}%

\global\long\def\Lbf{\mathbf{L}}%

\global\long\def\Ffrak{\mathfrak{F}}%
\global\long\def\Gfrak{\mathfrak{G}}%
\global\long\def\gfrak{\mathfrak{g}}%
\global\long\def\sfrak{\mathfrak{s}}%
\global\long\def\vfrak{\mathfrak{v}}%
\global\long\def\xibar{\bar{\xi}}%
\global\long\def\Cbb{\mathbb{C}}%
\global\long\def\Ebb{\mathbb{E}}%
\global\long\def\Fbb{\mathbb{F}}%
\global\long\def\Nbb{\mathbb{N}}%
\global\long\def\Rbb{\mathbb{R}}%
\global\long\def\extR{\widebar{\mathbb{R}}}%
\global\long\def\Pbb{\mathbb{P}}%
\global\long\def\Acal{\mathcal{A}}%
\global\long\def\Bcal{\mathcal{B}}%
\global\long\def\Ccal{\mathcal{C}}%
\global\long\def\Dcal{\mathcal{D}}%
\global\long\def\Fcal{\mathcal{F}}%
\global\long\def\Gcal{\mathcal{G}}%
\global\long\def\Hcal{\mathcal{H}}%
\global\long\def\Ical{\mathcal{I}}%
\global\long\def\Kcal{\mathcal{K}}%
\global\long\def\Lcal{\mathcal{L}}%
\global\long\def\Mcal{\mathcal{M}}%
\global\long\def\Ncal{\mathcal{N}}%
\global\long\def\Ocal{\mathcal{O}}%
\global\long\def\Pcal{\mathcal{P}}%
\global\long\def\Ucal{\mathcal{U}}%
\global\long\def\Scal{\mathcal{S}}%
\global\long\def\Tcal{\mathcal{T}}%
\global\long\def\Xcal{\mathcal{X}}%
\global\long\def\Ycal{\mathcal{Y}}%
\global\long\def\Ubf{\mathbf{U}}%
\global\long\def\bt{\mathbf{t}}
\global\long\def\bx{\mathbb{x}}
\global\long\def\Pbf{\mathbf{P}}%
\global\long\def\Ibf{\mathbf{I}}%
\global\long\def\Ebf{\mathbf{E}}%
\global\long\def\Abs{\boldsymbol{A}}%
\global\long\def\Qbs{\boldsymbol{Q}}%
\global\long\def\Lbs{\boldsymbol{L}}%
\global\long\def\Pbs{\boldsymbol{P}}%
\newcommand{\proj}{\textbf{proj}}
\newcommand{\prox}{\textbf{prox}}
\newcommand{\proximal}{\text{prox}}
\def\bx{\mathbf{x}}
\global\long\def\i{i}%
\global\long\def\Ibb{\mathbb{I}}

\DeclarePairedDelimiterX{\inprod}[2]{\langle}{\rangle}{#1, #2}
\DeclarePairedDelimiter\abs{\lvert}{\rvert}
\DeclarePairedDelimiter{\bracket}{ [ }{ ] }
\DeclarePairedDelimiter{\paran}{(}{)}
\DeclarePairedDelimiter{\braces}{\lbrace}{\rbrace}
\DeclarePairedDelimiterX{\gnorm}[3]{\lVert}{\rVert_{#2}^{#3}}{#1}
\DeclarePairedDelimiter{\floor}{\lfloor}{\rfloor}
\DeclarePairedDelimiter{\ceil}{\lceil}{\rceil}

\global\long\def\tsum{{\textstyle {\sum}}}

\newcommand{\opconex}{\texttt{OpConEx}}
\newcommand{\sopconex}{\texttt{S-OpConEx}}
\newcommand{\stochsopconex}{\texttt{S-StOpConEx}}
\newcommand{\adopconex}{\texttt{AdLagEx}}
\newcommand{\gconex}{\texttt{GradConEx}}
\newcommand{\stopconex}{\texttt{StOpConEx}}
\newcommand{\fstopconex}{\texttt{F-StOpConEx}}
\newcommand{\aconex}{\texttt{Aug-ConEx}}
\newcommand{\augconex}{\text{Aug-ConEx}}
\maketitle

\begin{abstract}
This paper proposes a novel first-order algorithm that solves composite nonsmooth and stochastic convex optimization problem with function constraints. 
Most of the works in the literature provide convergence rate guarantees on the average-iterate solution. There is growing interest in the convergence guarantees of the last iterate solution due to its favorable structural properties, such as sparsity
or privacy guarantees and good performance in practice. We provide the first method that obtains the best-known convergence rate guarantees on the last iterate for stochastic composite nonsmooth convex function-constrained optimization problems. Our novel and easy-to-implement algorithm is based on the augmented Lagrangian technique and uses a new linearized approximation of constraint functions, leading to its name, the Augmented Constraint Extrapolation ($\augconex$) method. We show that \augconex~achieves $\mathcal{O}(1/\sqrt{K})$ convergence rate in the nonsmooth stochastic setting without any strong convexity assumption and $\mathcal{O}(1/K)$ for the same problem with strongly convex objective function. While optimal for nonsmooth and stochastic problems, the \augconex~method also accelerates convergence in terms of Lipschitz smoothness constants to $\mathcal{O}(1/K)$ and $\mathcal{O}(1/K^2)$ in the aforementioned cases, respectively. To our best knowledge, this is the first method to obtain such differentiated
convergence rate guarantees on the last iterate for a composite nonsmooth stochastic setting without
additional $\log{K}$ factors. We validate the efficiency of our algorithm by comparing it with a state-of-the-art algorithm through numerical experiments. 
\vspace{1em}
\end{abstract}

\vspace{-1mm}
\section{Introduction}
We are interested in solving a convex function-constrained optimization problem
\begin{equation}\label{eq:prima-problem}
\begin{split}
    \min_{x\in X}\ \ &\psi_0(x) := f(x)+ \chi_0(x)\\
    \text{s.t.}\ \  &\psi_i (x):= g_i(x)+\chi_i(x)\leq 0 \quad \forall i  = 1,\cdots, m, 
\end{split}
\end{equation}
where $f:X\rightarrow \mathbb{R}$ and $g\added{_i}:X\rightarrow \mathbb{R}, {i = 1, \dots, m,}$ are either convex or strongly convex functions and $X$ is convex compact set. Moreover, $\chi_0$ is convex proper lower-semicontinuous function and $\chi_i = 1, \dots, m$ are continuous convex function. For ease of exposition, we use the notation $g$, $\chi$ and $\psi$ to denote the vector valued mapping associated with functions $g_i, \chi_i$ and $\psi_i$,  $i =1, \dots, m$, respectively. We allow both $f: X\rightarrow \mathbb{R}$ and $g: X\rightarrow \mathbb{R}^m $ to be {\em composite nonsmooth, i.e.,} $f$ and each component of $g$ are sum of smooth and nonsmooth functions. We also assume that $\chi_i$ for all $i = 0,1,\dots, m$ are ``simple'' functions in the sense that, for any fixed vector $v\in \mathbb{R}^n$ and non-negative weight vector $w\in \mathbb{R}^m$,  a certain proximal operator associated with the function $\chi_0(\cdot) + \sum_{i=1}^m w_i\chi_i(\cdot) + \inner{v}{\cdot}$ can be computed efficiently. We also let $f$ to be an expectation function $f(x) = \Ebb_\xi[\Ffrak(x, \xi)]$ where $\xi$ is the random data input to $\Ffrak$. We assume access to the first-order information of $f$ through a stochastic oracle which is unbiased and has bounded variance. Our main goal is to present easily implementable first-order methods with convergence rate guarantees on the last iterate for the composite nonsmooth stochastic convex optimization problem \eqref{eq:prima-problem}.

The constrained optimization problem of this form is often employed to impose a certain structure on the output solution. For example, in the field of sparsity-inducing optimization, one can use $\chi_0(x) = \gnorm{x}{1}{}$ whereas the constraint $\{x: g(x) \le 0\}$ enforces other operational requirements on $x$. One specific example of this type of problem includes sparse portfolio allocation, where $g$ can model the allocation of appropriate market capital to each sector that satisfies the real-world minimum and maximum limits imposed by legal requirements, risk constraints based on the investor's choice, and the normalization of the total allocation constraint. Note that the sparsity of the solution is an important metric for consolidating the number of investment assets.

There is a substantial amount of literature on solving convex function-constrained optimization. One approach uses the exact penalty, quadratic penalty, and augmented Lagrangian-based algorithms \cite{bertsekas99nonlinear, lan2013iteration, lan2016iteration, xu2021iteration, xu2020primal, li2021inexact}. A vast majority of this research typically applies to smooth and deterministic problems. Moreover, some methods involve solving complex penalty subproblems, leading to a two-loop algorithm structure with difficult subproblems to solve. The penalty subproblem can be further complicated if $f$ or $g$ have nonsmooth components, imposing high costs to retrieve their solution even when employing accelerated first-order methods. Hence, most works analyzing the total complexity of penalty-based methods deal with smooth functions $f$ and $g$ and assume that $\chi_i (x)  = 0$ for all $i =1,\dots,m$. Moreover, the works on stochastic penalty-based methods are relatively sparse \cite{xu2020primal}. Such consideration severely limits the applicability of penalty-based methods. However, if such restrictions are not an issue, a recent result shows that an augmented Lagrangian method converges at almost the same rate as unconstrained problems when the number of constraints is constant \cite{xu2022first}.

Another body of literature focuses on primal methods, such as cooperative subgradient \cite{Pol67, lan2020algorithms}, iterative regularization \cite{solodov2007explicit, helou2017e, kaushik2021method, samadi2023improved}, level-set methods \cite{lemarechal1995new, nesterov2018lectures, lin2018level, aravkin2019level} and moving ball type methods \cite{auslender2010moving, auslender2013very, nesterov2018lectures, boob2024level, singh2024stochastic} to solve \eqref{eq:prima-problem}. The level-set-based first-order methods require a root-finding procedure of a value function that composes $\psi_0$ and $\psi_i$ inside a max-function. Evaluation of such a value function is complicated when $f$ is a stochastic function. Hence, most works in the literature study the root-finding procedure for deterministic problems. Even in the deterministic setting, the root-finding procedure immediately leads to a complicated two-loop method with complex subproblems. Cooperative subgradient methods are easy-to-implement single-loop methods. However, they treat $f$ and $g$ as black-box nonsmooth functions and cannot directly accelerate the convergence in terms of their smooth component. Furthermore, the final convergence guarantee is provided for the average iterate only. The iterative regularization framework allows for only one constraint. Under this restriction, this framework places the constraint into the objective while introducing an iteration-dependent weight on the objective function. This revised objective does not have function constraints and can be solved using a single-loop method for unconstrained problems. The weighing scheme is changed systematically in every iteration to get convergence. However, such methods converge on the average iterate and suffer substantially in convergence rates compared to their unconstrained counterparts because of the constantly changing weight. Hence, they are suitable for ill-conditioned problems such as convex bilevel optimization or VI-constrained problems. Moving ball type methods are only applicable for problems smooth $f$ and $g$ since they require quadratic approximation of functions $f$ and $g$ to be tight. Moreover, these methods generate a QCQP subproblem, which requires implementing fast interior point type methods in the inner loop \cite{boob2024level}, giving rise to a two-loop scheme. More importantly, this approach does not work when the problem is a smooth composite, i.e., $f$ and $g$ are purely smooth and nonsmooth prox-friendly functions $\chi_i, i =0, \dots, m$ are present.

A major line of research focuses on primal-dual methods for smooth min-max saddle point problems independent of  \eqref{eq:prima-problem} \cite{Nemirovski2004, Nesterov2007, chambolle2011first, chen2014optimal, chambolle2016ergodic, hamedani2021primal, thekumparampil2022lifted, khalafi2023accelerated}. Note that if the strong-duality holds\footnote{Under the mild Slater assumption, one can show that strong duality holds \cite{bertsekas99nonlinear}}, one can reformulate \eqref{eq:prima-problem} into an equivalent saddle point problem which can be solved using primal-dual methods for non-bilinear coupling functions. However, the resulting saddle point problem is no longer uniformly smooth due to the unbounded domain of the dual variable. Hence, primal-dual methods require specialized and careful treatment for function-constrained problems. Such methods are considered in \cite{nedic2009approximate, yu2016primal, hamedani2021primal, boob2023stochastic, zhu2022new, zhang2022solving}. The majority of recently developed primal-dual methods are simple single-loop methods. They deal with the aforementioned problem of the unbounded set of dual variables in different ways. In particular, \cite{hamedani2021primal, zhu2022new} resolve this issue in the deterministic and smooth convex problems by adaptively searching for the step size and showing an explicit bound on the dual variable that holds throughout the run of the adaptive algorithm. \cite{zhang2022solving} implements a guess-and-check procedure for smooth and deterministic problems. Such techniques do not apply to nonsmooth or stochastic problems like those considered in this paper. Identifying this issue, \cite{boob2023stochastic} proposes a constraint extrapolation type design that directly deals with the unbounded dual set. Using this approach, they could address problem \eqref{eq:prima-problem}  in the composite nonsmooth stochastic setting as considered in this paper. However, the convergence result of almost all simple primal-dual methods is in terms of the average iterate, which destroys the desirable structural properties of the output solution. The only exception is \cite{zhu2022new}, who prove the last iterate convergence in the limited scope of smooth and deterministic problem \eqref{eq:prima-problem}. There also exists substantial recent interest in obtaining convergence rate guarantees for (stochastic) nonconvex optimization with function constraints. We refer the reader to \cite{boob2023stochastic, cartis2011evaluation,kong2019complexity, li2021augmented, lin2022complexity, boob2024level} and references therein. The best-known convergence rate guarantee to produce an $\epsilon$-KKT point is of $\mathcal{O}(\tfrac{1}{\sqrt{K}})$ under a well-known constraint qualification assumption \cite{boob2024level}. In this context, $K$ is the number of gradient evaluations.

There is a growing interest in algorithms converging in the last iterate for convex optimization problem \eqref{eq:prima-problem} without function constraints $\psi_i(x) \le 0,  i=1, \dots, m$. For such problems, when the objective is strongly convex and smooth with respect to the optimal solution, \cite{rakhlin2011making} shows a convergence rate of $\mathcal{O}(\frac{1}{K})$ where $K$ is the number of iterations of stochastic gradient descent. Note that the smoothness here requires existence of $L$ such that $f(x) -f(x^*) \le \frac{L}{2}\gnorm{x-x^*}{}{2}$ which is a stronger condition than the Lipschitz smoothness.  \cite{shamir2013stochastic, harvey2019tight} provide convergence rate of $\mathcal{O}(\frac{\log{K}}{\sqrt{K}})$ and $\mathcal{O}(\frac{\log{K}}{K})$  for nonsmooth stochastic convex and strongly convex problem, respectively. This result is almost optimal except for the presence of $\log{K}$ factor in the convergence rate. \cite{jain2021making} show that stochastic gradient descent can be optimal by appropriately modifying the step size scheme. In particular, they get the convergence rate of 
$\mathcal{O}(\frac{1}{\sqrt{K}})$ and $\mathcal{O}(\frac{1}{K})$ for nonsmooth stochastic convex and strongly convex optimization problems, respectively. The results of \cite{shamir2013stochastic, harvey2019tight, jain2021making} work for pure nonsmooth problems without function constraints, i.e., $f$ is a nonsmooth function and $\chi_0(x) = \psi_i (x) = 0$ for all $i =1, \dots, m$. Very recently, \cite{liu2023revisiting} shows such results can be extended to composite nonsmooth problems while maintaining the optimal convergence rate guarantees of \cite{jain2021making} for both convex and strongly convex problems. They also showed that if the stepsizes are properly selected, one can accelerate the convergence rate in the smooth component of $f$. It is also worth mentioning the recent progress on incremental gradient-type methods with convergence on the last-iterate \cite{cai2024last}. 
A substantial limitation of this entire body of literature on the last iterate convergence is that they all require projection onto the constraint set. This assumption may be challenging to satisfy when complicated function constraints such as $\psi_i(x) \le 0$ in \eqref{eq:prima-problem} are present. 

Many of the studies mentioned earlier for the last iterate convergence are motivated by the fact that the usual information-theoretically optimal convergence guarantees for convex and strongly convex problems \cite{polyak1992acceleration, nemirovskij1983problem, zinkevich2003online, deng2018optimal} hold for some average iterate. On the other hand, it is common in practice to consider the final iterate as the solution \cite{shalev2007pegasos, shamir2012averaging}. In recent years, many works have shown other hidden benefits of last-iterate convergence on differential privacy. In particular, instead of revealing the entire history of iterates visited through the average iterate, if one only reveals the last iterate obtained through stochastic gradient type methods, then it leads to substantially stronger differential privacy guarantees \cite{feldman2018privacy, altschuler2022privacy, asoodeh2023privacy}. 

Clearly, there is a significant amount of literature on the two independent streams of research: the convex function constrained optimization and the last iterate convergence of first-order methods. Despite such vital developments, to the best of our knowledge, no algorithm in the literature provides convergence rate guarantees in the last iteration for stochastic composite-nonsmooth convex function-constrained optimization problems. We intend to contribute to this area in this paper. In particular, we present an Augmented Constraint Extrapolation (\augconex) method. This method aims to solve a min-min-max problem associated with a specific Lagrangian formulation of \eqref{eq:prima-problem} that introduces primal slack variables in the constraint. We also augment a quadratic penalty of the resulting constraint function inside the Lagrangian formulation. Finally, to deal with unbounded Lagrange multipliers and slack variables, we introduce a novel linearization scheme of the constraint functions $g_i, i =1, \dots, m$, used inside an extrapolation term. Hence, we call the resulting algorithm the Augmented Constraint Extrapolation method. Apart from the novel design of this algorithm, which may be of independent interest, this method exhibits convergence guarantees in the last iterate that matches the corresponding guarantees in the average iterate for problem \eqref{eq:prima-problem} or the convex composite-nonsmooth problems without function constraints. Note that the composite-nonsmooth function class contains Hölder-smooth functions as a specific subclass \cite[Proposition 2.1]{liang2024unified}. Hence, our setting strictly generalizes pure smooth, Hölder smooth, and pure nonsmooth problems - covering a wide range of problems.  

Table \ref{tab:comparison} provides a brief summary of the aforementioned results from the literature and puts our contribution in the context. Throughout this paper, we follow the well-known criterion of convergence for problem \eqref{eq:prima-problem} which requires both the optimality gap and infeasibility of the candidate solution to be small simultaneously. In particular, when the objective is strongly convex, \augconex~ method converges at the rate of $\mathcal{O}(\tfrac{1}{K})$ (i.e., optimality gap and infeasibility are of $\mathcal{O}(\tfrac{1}{K})$), where $K$ is the number of stochastic gradients of $f$, and number of gradients of each $g_i$. It is worth noting that this $\mathcal{O}(\tfrac{1}{K})$ term is due to nonsmooth components of functions $f$ and $g$, as well as errors associated with the stochastic oracle of $f$. The method exhibits the accelerated $\mathcal{O}(\frac{1}{K^2})$ convergence in terms of smooth components of the problem.  When the problem is convex, the \augconex~method converges at the rate of $\mathcal{O}(\tfrac{1}{\sqrt{K}})$ in terms of nonsmooth components of $f$ and $g$, as well as stochastic errors due to $f$. In terms of smooth components, the method exhibits accelerated convergence rate of $\mathcal{O}(\tfrac{1}{K^2})$. This type of composite convergence rate results of the last iterate are state-of-the-art even when compared with literature on convex problems without function constraints. Indeed, \cite{liu2023revisiting} have dependence on $\log{K}$-factor in the convergence rate when dealing with composite nonsmooth problems. 
\begin{table}[h]
        \centering
        \begin{tabular}{c|c|c|c|c|c|c}
            \hline
             \multirow{3}{1.7cm}{Algorithms}& \multicolumn{4}{c|}{Setting}&\multicolumn{2}{c}{Convergence rate}\\
            \cline{2-7}
             & \multirow{2}{*}{Stochastic} & \multirow{2}{*}{Last-iterate} & Composite-  &Function- & \multirow{2}{*}{Convex} & Strongly \\
             & & &nonsmooth$^\ast$ &constraint & &convex\\
            \hline
            \multirow{2}{*}{ConEx\cite{boob2023stochastic}}  &\multirow{2}{*}{$\checkmark$ }& \multirow{2}{*}{$\times$} & \multirow{2}{*}{$\checkmark$ } & \multirow{2}{*}{$\checkmark$} & \multirow{2}{*}{$\mathcal{O}\big(\frac{1}{\sqrt{K}}\big)$}& \multirow{2}{*}{$\mathcal{O}\big(\frac{1}{K}\big)$}\\
            &&&&&& \\\hline
            Primal- &\multirow{2}{*}{$\times$} & \multirow{2}{*}{$\checkmark$} & \multirow{2}{*}{$\times$} &\multirow{2}{*}{$\checkmark$} & \multirow{2}{*}{$\mathcal{O}\big(\frac{1}{K}\big)$} &\multirow{2}{*}{$\mathcal{O}\big( \frac{1}{K^2}\big)$}\\
            Dual \cite{zhu2022new} & & & & & & \\\hline
            Stochastic & \multirow{2}{*}{$\checkmark$} &\multirow{2}{*}{$\checkmark$} &\multirow{2}{*}{$\times$} &\multirow{2}{*}{$\times$} &\multirow{2}{*}{$\mathcal{O}\big( \frac{1}{\sqrt{K}}\big)$} &\multirow{2}{*}{$ \mathcal{O} \big( \frac{1}{K}\big)^\dagger$}\\
            Gradient \cite{jain2021making} & & & & & &\\ \hline
            Stochastic & \multirow{2}{*}{$\checkmark$} &\multirow{2}{*}{$\checkmark$} &\multirow{2}{*}{$\checkmark$} &\multirow{2}{*}{$\times$} &\multirow{2}{*}{$\mathcal{O}\big( \frac{\log{K}}{\sqrt{K}}\big)^\dagger$} &\multirow{2}{*}{$ \mathcal{O} \big( \frac{\log{K}}{K}\big)^\dagger$}\\
            Gradient \cite{liu2023revisiting} & & & & & &\\ \hline
            \multirow{2}{1.7cm}{\centering$\augconex$\\(this work) } & \multirow{2}{*}{$\checkmark$} &\multirow{2}{*}{$\checkmark$} &\multirow{2}{*}{$\checkmark$} &\multirow{2}{*}{$\checkmark$} &\multirow{2}{*}{$\mathcal{O}\big( \frac{1}{\sqrt{K}}\big)$} &\multirow{2}{*}{$ \mathcal{O} \big( \frac{1}{K}\big)$}\\
            & & & & & & \\ \hline
        \end{tabular}
        \caption{
        Comparison of {state-of-the-art} algorithms with $\augconex$ method. $K$ denotes the number of (stochastic) gradient evaluations of $f$ and $g$. \\
        $\ast$ Not being composite nonsmooth refers to the problem being purely smooth.\\
        $\dagger$ The $\log{K}$ factor in the convergence rate can be removed in the purely nonsmooth case.
        }
        \label{tab:comparison}
    \end{table}
 \subsection{Outline} 
 In Section \ref{sec:notation}. we describe notation and terminologies. Section \ref{sec:formulation} discusses the preliminaries and fundamental assumptions. Section \ref{sec:algorithm} comprehensively illustrates the $\augconex$ method for solving problem \eqref{eq:prima-problem}. Section \ref{sec:convegence-analysis-augconex} presents unified convergence analysis of $\augconex$ method in both convex and strongly convex cases. The numerical experiments are shown in Section \ref{sec:numerical} and concluding remarks are mentioned in Section~\ref{sec:conclusion}. 
\subsection{Notation and Terminologies}\label{sec:notation}
In this paper, we use the following notations. Let $[m]:=\{1,\dots,m\}, \psi(x):= [\psi_1(x), \dots, \psi_m(x)]^\top, g(x):= [g_1(x),\dots,g_m(x)]^\top$ and $\chi:= [\chi_1(x),\dots,\chi_m(x)]^\top$ and the constraint in \eqref{eq:prima-problem} can be expressed as $\psi(x)\leq \textbf{0}$ where the bold $\textbf{0}$ denotes the vector of elements $0$. Here, the size of the vector is not specified and depends on the context. Also, $\mathbb{R}_+^k$ and $\mathbb{R}_-^k$ represent non-negative and non-positive orthants of a $k$-dimensional Euclidean space, respectively. Euclidean norm is denoted as $\gnorm{\cdot}{}{}$ and the standard inner product is defined as $\inner{\cdot}{\cdot}$. Further, $\gnorm{\cdot}{1}{}$ denote the $\ell_1$ norm. 
Let $\mathcal{B}^2(r):= \{x:\gnorm{x}{}{}\leq r\}$ be the Euclidean ball of radius $r$ centered at origin. For a convex set $X$, we denote the normal cone at a point $x\in X$ as $N_X(x)$. The set of interior points and relative interior points of the set $X$ are defined as $\inter{X}$ and $\rinter{X}$, respectively. $I_X$ denotes the indicator function of the set $X$. For any proper, closed, and convex function $r:\mathbb{R}^k \rightarrow \mathbb{R}\cup \{+\infty\}$ and $\dom(r): = \{x\in \mathbb{R}^k|r(x)<+\infty\} $ as its effective domain, $r^*(y): = \sup_{x}\{\inner{x}{y}-r(x)\}$ denotes its Fenchel conjugate. The $``+"$ operation on sets denotes the Minkowski sum. $[x]_+:= \max\{x,0\} $ and $[x]_- := \min\{x,0\}$ for any $x\in \mathbb{R}$. For any vector $x\in \mathbb{R}^k$, we denote $[x]_+$ and $[x]_-$ as element-wise application of the operator $[\cdot]_+$ and $[\cdot]_-$, respectively. The $i$-th element of vector $x$ is denoted as $x^{i}$. A function $r(\cdot)$ is $\lambda$-\textit{Lipschitz smooth} if the gradient $\grad r(x)$ is $\lambda$-Lipschitz continuous, i.e.,  for some $\lambda\geq 0$, we have 
\[
\gnorm{\grad r(x) - \grad r(y)}{}{} \le \lambda\gnorm{x-y}{}{},\quad \forall x,y\in \dom(r).
\]
This implies
\[
 r(x)-r(y)-\inner{\grad r(y)}{x-y}\leq \tfrac{\lambda}{2}\gnorm{x-y}{}{2}, \quad \forall x,y\in \dom(r).
\]   
In many cases, it is possible that a convex function $ r $ is a combination of Lipschitz smooth and nonsmooth functions. We call such functions as composite-nonsmooth satisfying 
\[r(x)-r(y)-\inner{\grad r(y)}{x-y}\leq \tfrac{\lambda}{2}\gnorm{x-y}{}{2} + \omega \gnorm{x-y}{}{}, \quad \forall x,y\in \dom(r).\]
Here, when $\lambda = 0$, we call $r$ a purely nonsmooth function. Moreover, when $\omega = 0$, it results in the usual purely smooth function $r$. 
We define the diameter of the compact set $X$ by 
\begin{equation}\label{eq:diam}
	D_X: = \max_{x,y\in X}\tfrac{\|x-y\|}{2}.
\end{equation}
For any convex function $ h $, we denote the subdifferential as $\partial h $ as follows: at a point $ x \in \rinter{X} $,  $\partial h $ is comprised of all subgradients $ h' $ of $ h $ at $ x $ which are in the linear span of $ X - X $. For a point $ x \in  X\backslash \rinter{X} $, the set $\partial h(x) $ consists of all vectors $ h' $, if any, such that there exists $ x_i \in  \rinter{X} $ and $ h'_i \in \partial h(x_i), i=1,2,\dots,$ with $x = \lim_{i\rightarrow \infty} x_i, h'  = \lim_{\i\rightarrow \infty }h'_i. $ With this definition, we have the following relation for a $\mathcal{M}$-smooth convex function with respect to norm $\gnorm{\cdot}{}{}$ for any $x\in X$ whose set $\partial h(x)$ is non-empty
\[
h'\in \partial h(x)\Rightarrow \|\inner{h'}{d}\|\leq \mathcal{M}\|d\|,\quad \forall d \in \text{lin}(X-X),
\]   
implying that 
\[
h'\in \partial h(x)\Rightarrow \|h'\|\leq \mathcal{M}.
\]
A function $r(\cdot)$ is $\beta$-strongly convex if 
\begin{equation}\label{eq:strong-convex}
	r(x)\geq r(y) + \inner{r'(y)}{x-y} + \tfrac{\beta}{2}\|x-y\|^2, \forall x,y\in X.
\end{equation}
We denote the strong-convexity parameter of $f$ as $\mu_f \ge 0$. When $\mu_f = 0$, we have that $f$ is merely convex. Otherwise, $f$ is $\mu_f$-strongly convex.

\section{Preliminaries}\label{sec:formulation}
Let $\Lcal(x,y)$ be the associated Lagrangian function of problem \eqref{eq:prima-problem}. Therefore, one can write the corresponding saddle point problem as follows
\begin{equation}\label{eq:sp-problem}
	\min_{x\in X}\max_{y\in \mathbb{R}^m_+}\{ \Lcal(x,y): = \psi_0(x)+ \inner{y}{\psi(x)} \},
\end{equation}
where $\Lcal(x,y)$ is convex in $x$ for any $y\in \mathbb{R}_+^m$ and linear in $y$ for any $x\in X$.
First, we assume that there exist solutions for \eqref{eq:sp-problem} and its corresponding dual problem. In particular, we have the following assumption
\begin{assumption}\label{assump:SPP}
	There exists $(x, y )\in  X\times \mathbb{R}^m_+$  of $\Lcal(x,y)$ such that
	\begin{equation}\label{eq:SPP-assump}
		\Lcal(x^*,y)\leq \Lcal(x^*,y^*)\leq \Lcal(x,y^*) \quad \forall (x,y)\in X\times \mathbb{R}^m_+.
	\end{equation}
	Moreover, we assume $\Lcal(x^*,y^*)$ is finite. 
\end{assumption} 
This assumption is mild and is widely used in the literature on dual-based methods. In particular, the existence of a Slater point is sufficient to ensure the existence of $y^*$\cite{bertsekas99nonlinear}. Existence of $x^*$ follows by the compactness of $X$. Furthermore, due to strong duality, $x^*$ is the optimal solution of \eqref{eq:prima-problem}. Throughout this section, we assume the existence of $y^*$ satisfying \eqref{eq:SPP-assump}. Denote the optimal value $\psi_0^*:= \psi_0(x^*)$. Then the following definition illustrates a commonly used optimality measure for problem \eqref{eq:prima-problem}. 
\begin{definition}\label{def:gaps}
	A point $\widebar{x}\in X $ is called $(\epsilon_o,\epsilon_c)$-optimal solution of the problem \eqref{eq:prima-problem} if 
	\begin{equation*}
		\psi_0(\widebar{x}) - \psi_0^*\leq \epsilon_o\quad \text{and}\quad \gnorm{[\psi(\widebar{x})]_+}{}{}\leq \epsilon_c.
	\end{equation*}
	Consequently, a stochastic $(\epsilon_o,\epsilon_c)$-approximately optimal solution satisfies 
	\begin{equation*}
		\Ebb[\psi_0(\widebar{x}) - \psi_0^*]\leq \epsilon_o\quad \text{and}\quad \Ebb[\gnorm{[\psi(\widebar{x})]_+}{}{}]\leq \epsilon_c.
	\end{equation*}
\end{definition}
Moreover, $f$ can be a stochastic function covering a wide variety of data-driven convex optimization problems. In particular, we assume the objective function $f$ is an expectation function, whose first-order information can be approximated by an unbiased estimator obtained through a stochastic oracle $\mathfrak{F}(x,\xi)$ for any $x \in X$. In particular, we have
\begin{equation}\label{eq:SO_F_objective}
	\forall x \in X, \quad  f'(x) = \Ebb[\Ffrak(x,\xi)], \quad \gnorm{\Ffrak(x, \xi) - f'(x)}{}{2} \le \sigma^2,
\end{equation} 
where $\xi$ is a random variable that indicates the source of uncertainty and is independent of the search point $x$. 
Next, we provide a new Lagrangian function $\Lcal(x,s,y)$ by introducing slack variables $-s^{i},i\in [m]$ to the constraints of problem \eqref{eq:prima-problem} as follows 
\begin{equation}\label{eq:sp-problem-conj}
 \Lcal(x,s,y): = \psi_0(x)+ \inner{y}{\psi(x)-s}. 
\end{equation}
	We explain the reason for defining this function in Section \ref{sec:algorithm} in detail. One immediate result of \eqref{eq:sp-problem-conj} is the following relation 
\begin{equation}\label{eq:upperbound}
	\Lcal(x,y)\leq \Lcal(x,s,y)\quad  \forall (x,s,y)\in X\times  \mathbb{R}_-^m\times  \mathbb{R}_+^m\quad \text{and}\quad 	\Lcal(x,y)= \Lcal(x,0,y)\quad \forall (x,y)\in X\times\mathbb{R}_+^m   .
\end{equation}

As mentioned earlier, we assume that $\chi_i, i\in [m]$ are ``simple'' functions such that for any scalar $\eta >0$, vector $v\in \mathbb{R}^n$ and $w\in \mathbb{R}^m_+$, we can efficiently compute the following \textbf{prox} operator  
	\begin{equation}\label{eq:prox}
		\prox (w,v,\tilde{x},\eta): = \argmin_{x \in X}\{\chi_0(x) + w^\top \chi + \inner{v}{x} + \tfrac{\eta}{2}\gnorm{x-\tilde{x}}{}{2}\}.
	\end{equation}
To avoid any confusion, we define the classical proximal operator of a function $h(\cdot)$ at point $\wb{x}$, $\proximal_h(\wb{x})$ as follows 
	\begin{equation}\label{eq:prox-1}
		\proximal_{h}(\wb{x}) = \argmin_{x}\{h(x)+ \tfrac{1}{2}\gnorm{x-\wb{x}}{}{2}\}.
	\end{equation}
Then, in view of \eqref{eq:prox} and \eqref{eq:prox-1}, it is easy to see that $\prox(w,v,\tilde{x}, \eta) \equiv \proximal_{\frac{1}{\eta}(\chi_0(\cdot) + w^\top \chi(\cdot) + \inprod{v}{\cdot}) + I_X(\cdot)}(\wt{x})$.

\section{The Aug-ConEx Method  } \label{sec:algorithm}
 Before stating the algorithm, let us delve into important details. First, we assume that $f$ and $g_i$, $i\in[m]$  satisfy the following
\begin{equation}\label{eq:smooth-nonsmooth}
	f(x)-f(y) - \inner{f'(y)}{x-y} \le \tfrac{L_f}{2}\gnorm{x-y}{}{2} + H_f\gnorm{x-y}{}{}, \quad \forall x, y\in X,\quad\forall f'(y)\in \partial f(y),
\end{equation}
where $L_f$ and $H_f$ are the smoothness and non smoothness constants of $f$. For constraints, we make a similar assumption as in \eqref{eq:smooth-nonsmooth} 
\begin{equation}\label{eq:sonstraint-smooth-nonsmooth}
	g_i(x)-g_i(y) - \inner{g_i'(y)}{x-y} \le \tfrac{L_{g_i}}{2}\gnorm{x-y}{}{2} + H_{g_i}\gnorm{x-y}{}{}, \quad \forall x, y\in X,\quad\forall g_i'(y)\in \partial g_i(y), i\in [m],
\end{equation}
where $L_{g_i}$ and $H_{g_i}$ are smoothness and non smoothness constants of the functions $g_i$ for all $i\in[m]$. Due to \cite[Proposition 2.1]{liang2024unified}, we note that the class of functions defined in \eqref{eq:smooth-nonsmooth}, \eqref{eq:sonstraint-smooth-nonsmooth} contain Hölder smooth functions as a subclass. Hence, our problem setting contains pure smooth, pure nonsmooth, and Hölder smooth settings as specific cases. In addition, we assume that the constraint functions are Lipschitz continuous. Specifically, we have the following
\begin{subequations}\label{eq:M_g}
	\begin{equation}
		g_i(x)-g_i(y)\leq M_{g_i}\gnorm{x-y}{}{} \quad \forall x,y\in X\quad i\in [m],
	\end{equation}
	\begin{equation}
		\chi_i(x)-\chi_i(y)\leq M_{\chi_i}\gnorm{x-y}{}{} \quad \forall x,y\in X\quad i \in [m].
	\end{equation}
\end{subequations}
The Lipschitz-continuity assumption in \eqref{eq:M_g} is common in the literature when $g_i$, $i\in [m]$, are nonsmooth functions. If $g_i$, $i\in [m]$, are Lipschitz smooth then their gradients are bounded due to the compactness of $X$. Hence \eqref{eq:M_g} is not a strong assumption for the given setting. Also note that due to the definition of subgradient for the convex function in Section \ref{sec:notation}, we have $\gnorm{g'_i(\cdot)}{}{}\leq M_{g_i}$ which implies $\gnorm{g'_i(y)^\top (x-y)}{}{}\leq \gnorm{g'_i(y)}{}{}\gnorm{x-y}{}{}\leq M_{g_i} \gnorm{x-y}{}{}$. Therefore, using the relations in \eqref{eq:sonstraint-smooth-nonsmooth} and \eqref{eq:M_g}, one can state the following assumption for $g = [g_1,\cdots,g_m]^\top$ and $\chi =  [\chi_1,\cdots,\chi_m]^\top$.
\begin{subequations}\label{eq:M_g_1}
	\begin{equation}
		\gnorm{g(x)-g(y)}{}{}\leq M_{g}\gnorm{x-y}{}{} \quad \forall x,y\in X,
	\end{equation}
	\begin{equation}
		\gnorm{\chi(x)-\chi(y)}{}{}\leq M_{\chi}\gnorm{x-y}{}{} \quad \forall x,y\in X,
	\end{equation}
	\begin{equation}\label{eq:sonstraint-smooth-nonsmooth-1}
		\gnorm{g(x)-g(y) - \inner{g'(y)}{x-y}}{}{} \le \tfrac{L_{g}}{2}\gnorm{x-y}{}{2} + H_{g}\gnorm{x-y}{}{}\quad \forall x, y\in X,
	\end{equation}
	\begin{equation}
		\gnorm{g'(y)^\top (x-y)}{}{}\leq M_g\gnorm{x-y}{}{}\quad \forall x,y\in X,
	\end{equation}
\end{subequations}
where $g'(\cdot):=[g'_1(\cdot),\dots, g'_m(\cdot)] \in \mathbb{R}^{n\times  m}$, $M_g  := (\sum_{i=1}^{m}M_{g_i}^2)^{1/2}\text{, }M_{\chi} := (\sum_{i=1}^{m}M_{\chi_i}^2)^{1/2} $, $L_g  := (\sum_{i=1}^{m}L_{g_i}^2)^{1/2}$ and $H_g  := (\sum_{i=1}^{m}H_{g_i}^2)^{1/2}$. As mentioned earlier, we denote $\mu_f$ as the modulus of strong convexity for $f$. We say that the convex ($\mu_f = 0$)or strongly convex ($\mu_f > 0-$) problem \eqref{eq:prima-problem} is deterministic composite nonsmooth if we can evaluate exact subgradients of $f$. In that case, we have $\sigma = 0$ in \eqref{eq:SO_F_objective}. Otherwise, the it is stochastic composite nonsmooth problem. 
In this paper, we denote $\ell_g(x_{t+1},\hat{x}_t)$ a linear approximation of $g(\cdot)$ at point $\hat{x}_t$, i.e., 
\begin{equation}\label{eq:ell_g}
	\ell_g(x_{t+1};\hat{x}_t) : =  g(\hat{x}_t)+\inner{g'(\hat{x}_t)}{x_{t+1}-\hat{x}_t}.
\end{equation}
We use this quantity regularly throughout our algorithm and analysis. Hence, for the ease of notation, we denote $ \ell_g(x_{t+1};\hat{x}_t)$ as $\ell_g(x_{t+1})$ and drop the linearization point $\wh{x}_t$.

We are now ready to introduce $\augconex$ for stochastic composite nonsmooth convex program \eqref{eq:prima-problem} (see Algorithm \ref{alg:augConEx}). First, the objective function $f$ is evaluated by a stochastic oracle $\mathfrak{F}$ at point $\hat{x}_k$ using random input $\xi_k$. Second, the updates for $\wh{x}_{k+1}, \wt{y}_{k+1}$ and $\wb{y}_{k+1}$ in \eqref{eq:Denotes} are quite simple to compute. Notice that the update in \eqref{eq:conceptual-x-s-update} requires solving for $x_{k+1}$ and $s_{k+1}$ simultaneously. Indeed, the definition of $y_{k+1}$ and $\tilde{V}_{k+1}$ in \eqref{eq:Denotes} implies $x_{k+1}$ depends on $s_{k+1}$. Moreover, $s_{k+1}$ clearly depends on $x_{k+1}$. At this point, it is unclear whether the step in \eqref{eq:conceptual-x-s-update} can be computed efficiently. Later, we answer this question with a resounding yes. In particular, we show a simple proximal operator fixed point iteration that efficiently obtains $(s_{k+1},x_{k+1})$ without any new evaluation of $\mathfrak{F}$, $g'$ or $g$. We call this method efficient since it converges to its unique fixed point $x_{k+1}$ linearly. Section \ref{sec:implicit-update} discusses implementing the update in \eqref{eq:conceptual-x-s-update} in detail. 
\begin{algorithm}[t]
	\caption{Conceptual \textbf{Aug}mented \textbf{Con}straint \textbf{Ex}trapolation method (\aconex)}\label{alg:augConEx}
	\begin{algorithmic}[1]
		\State {\bf Input:} $x_1 \in X$, $y_1$, $\{\rho_{1},\tau_1,\eta_{1},L_1\}$, $K$
		\State Set $\wh{x}_1 \gets x_1$, $\wt{y}_1 \gets y_1$, $V_1 \gets g(x_1)+ \chi(x_1)$
		\For{k = $1, \dots, K-1$}
		\State Update $(s_{k+1},x_{k+1})$ satisfying
		\begin{equation}\label{eq:conceptual-x-s-update}
			\begin{split}
				&s_{k+1} \gets\argmin_{s \le \mathbf{0}} \inprod{-\wt{y}_k}{s} + \tfrac{\rho_k}{2}\gnorm{\ell_g(x_{k+1})+ \chi(x_{k+1}) - (1-\tau_k)V_k -s}{}{2},\\
				&x_{k+1} \gets \prox(y_{k+1},\mathfrak{F}(\hat{x}_k,\xi_k) +  g'(\wh{x}_k) y_{k+1},\hat{x}_k,L_k ),\\
			\end{split}
		\end{equation}
		\begin{equation}\label{eq:Denotes}
			\textbf{Set }\begin{cases}
				V_{k+1} \gets& g(x_{k+1})+ \chi(x_{k+1}) - s_{k+1},\\
				\wt{V}_{k+1} \gets& \ell_g(x_{k+1})+ \chi(x_{k+1}) - s_{k+1},\\
				y_{k+1}\gets& \wt{y}_k + \rho_k(\wt{V}_{k+1}-(1-\tau_{k})V_k),\\
				\wt{y}_{k+1} \gets& \wt{y}_k + \eta_k[\wt{V}_{k+1} - (1-\tau_k)V_k],\\
				\wb{y}_{k+1} \gets& (1-\tau_k)\wb{y}_k + \tau_k y_{k+1},\\
				\wh{x}_{k+1} \gets& x_{k+1} + \beta_{k+1}(x_{k+1}-x_k).\\
			\end{cases}
		\end{equation}
		\EndFor\\
		\Return $x_K$.
	\end{algorithmic}
\end{algorithm}
Before providing the precise statement on the computation of \eqref{eq:conceptual-x-s-update} and discussing the convergence guarantees of Algorithm \ref{alg:augConEx}, we elaborate on the motivation for this design as well as the introduction of the new variable $s \le \zero$ inside the \augconex~method. For ease of exposition, we remove the components $\chi_i, i = 0, \dots,  m$ throughout this discussion. Then, first note that the Lagrangian reformulation of \eqref{eq:prima-problem} is 
\begin{align*}
    &\min_{x\in X} \max_{y \ge \zero} f(x) + y^\top g(x) \equiv \min_{x\in X} \max_{y} f(x) + y^\top g(x) - h(y) ,
\end{align*}
where $h$ is the indicator function of the nonnegative orthant $\{y \ge \zero\}$. It is easy to see that the conjugate of function $h$, represented as $h^*$, is the indicator of the nonpositive orthant. Then, in the above Lagrangian formulation, we write $h$ in its conjugate form as 
\begin{align*}
    \min_{x \in X} \max_{y} f(x) + y^\top g(x) - \sup_{s}[s^\top y - h^*(s)]  &\equiv \min_{x \in X} \max_{y} f(x) + y^\top g(x) - \sup_{s \le \zero} s^\top y \\
    &\equiv \min_{x \in X} \max_y \inf_{s\le \zero} f(x) + y^\top [g(x)-s] \\
    &\equiv \min_{x \in X} \min_{s \le \zero} \max_{y} f(x) + y^\top [g(x)-s],
\end{align*}
where first equivalence follows since $h^*(s)$ is the indicator of $\{s \le \zero \}$, third equivalence follows due to strong duality. The resulting min-min-max problem can be reimagined as the introduction of the slack variable $-s$ inside the constraint $g(x) \le 0$. Finally, we add a penalty for the constraint $g(x) - s$ and rename the unconstrained $y$ variable as $\wt{y}$ to avoid confusion with the nonnegative Lagrange multiplier in the first reformulation to get the following augmented Lagrangian form
\begin{equation}\label{eq:s_opt_prob}
\min_{x\in X} \min_{s \le 0} \max_{\wt{y}} [L_\rho(x,s,\wt{y}) := f(x) + \wt{y}^\top [g(x)-s] + \frac{\rho}{2} \gnorm{g(x)-s}{}{2}].    
\end{equation}
Note that such a form of augmented Lagrangian is already used in the context of compositional optimization in \cite{zhu2022new}. For a given $(\wt{y}_k, x_{k+1})$ and penalty parameter $\rho_k$, the minimizer $s \le \zero$ is essentially 
\[\argmin_{s \le \zero} -\wt{y}^\top s + \frac{\rho_k}{2} \gnorm{g(x_{k+1})-s}{}{2}.\]
In $s_{k+1}$ update of \eqref{eq:conceptual-x-s-update}, we modify the above step with a linearization of $g(x_{k+1})$ as $\ell_g(x_{k+1})$ and add $-(1-\tau_k)(g(x_k)-s_k)$ which is effectively the constraint evaluation in the previous iteration. It is not difficult to see that $y_{k+1}$ is the negative of the gradient at $s_{k+1}$ for the optimization problem in \eqref{eq:s_opt_prob} Hence, by optimality condition of $s_{k+1}$, we must have $y_{k+1} \ge \zero$. We treat $y_{k+1}$ as the nonnegative Lagrange multiplier. Since $s$ is the conjugate variable of $y$, the expression of $y_{k+1}$ as the gradient at $s_{k+1}$ is a natural choice. We use the Lagrange multiplier $y_{k+1}$ in the update of $x_{k+1}$ and perform a proximal gradient update over set $X$. In particular, expanding $\prox(y_{k+1}, \Ffrak(\wh{x}_k, \xi_k) + g'(\wh{x}_k)y_{k+1}, \wh{x}_k, L_k)$, we have
\[x_{k+1} = \argmin_{x \in X} \inprod{\Ffrak(\wh{x}_k, \xi_k) + g'(\wh{x}_k) y_{k+1}}{x} + \chi_0(x) + \sum_{i = 1}^m y_{k+1}^{(i)}\chi_i(x) + \frac{L_k}{2} \gnorm{x-\wh{x}_k}{}{2}.\]
This is a legitimate proximal update for the function $\Lcal(\cdot, y_{k+1})$.
The  gradient ascent update for $\wt{y}_{k+1}$ with step size parameter $\eta_k$ yields $\wt{y}_{k+1} = \wt{y}_k + \eta_k[g(x_{k+1}) -s_{k+1}]$. We modify this step with a linearization of $g$ as $\ell_g(x_{k+1})$ and introducing the $-(1-\tau_k)(g(x_k)-s_k)$. The  combination terms  of linearization $\ell_g$ and constraint evaluation $(1-\tau_k)V_k$ are novel design elements that are important for the convergence of the last iterate in presence of unbounded variables $\wt{y}$ and $s$. Finally, the extrapolation 
\[\wh{x}_{k+1} = x_{k+1} + \beta_{k+1}(x_{k+1}-x_k),\]
is reminiscent of Nesterov's acceleration which is important to get a fast convergence guarantee. Indeed, it is not hard to verify that if $\psi(x) = g(x) = \zero$ and $\chi_0(x) = 0$, then the resulting \augconex~method has $s_k = \wt{y}_k = y_k = \zero$ and it reduces to the well-known Accelerated Gradient Descent method \cite{nesterov83, nesterov2018lectures}. 
We now focus our attention to the evaluation of the update in \eqref{eq:conceptual-x-s-update}.
\subsection{Computation of the $(x,s)$-update in \augconex~Method}\label{sec:implicit-update}
In this section, we show a computational scheme for evaluating the update in \eqref{eq:conceptual-x-s-update} in Algorithm \ref{alg:x-s-update} below. We first define the following $ \prox $ operator which is used inside Algorithm \ref{alg:augConEx}.
\begin{equation}\label{eq:prox-operaotr1}
	F(\bx) := \prox \paran[\Big]{\rho_k\bracket[\big]{U_k + g'(\hat{x}_k)^\top \bx  + \chi (\bx)}_+,\mathfrak{F}(\wh{x}_k,\xi_k) + \rho_k g'(\wh{x}_k)[U_k + g'(\wh{x}_k)^\top \bx+\chi (\bx)]_+, \wh{x}_k, L_k},
\end{equation}

where $U_k = g(\wh{x}_k) - g'(\wh{x}_k)^T\wh{x}_k - (1-\tau_k)V_k + \tfrac{\wt{y}_k}{\rho_k}$. Equivalently, one can use the classical definition of $\proximal$ operator and rewrite  $F(\bx)$ in \eqref{eq:prox-operaotr1} as follows 
\begin{equation}\label{eq:prox-operaotr2}
	F(\bx) = \proximal_{\Gamma(x;\bx)/L_k}\paran[\Big]{\wh{x}_k - \tfrac{1}{L_k}\braces[\big]{\mathfrak{F}(\wh{x}_k,\xi_k) + \rho_k g'(\wh{x}_k)[U_k + g'(\wh{x}_k)^\top \bx+\chi (\bx)]_+}},
\end{equation}
where $\Gamma(x;\bx) = I_X(x)+\chi_0(x) + \inner{\chi(x)}{c(\bx)}$, $c(\bx) = \rho_k [U_k + g'(\wh{x}_k)^\top \bx + \chi(\bx)]_+$ and $I_X(x)$ is the indicator function of set $X$. Note that $c(\bx):\mathbb{R}^n\rightarrow \mathbb{R}^m$ is a mapping independent of $x$. 
\begin{algorithm}[H]
	\caption{Compute update in \eqref{eq:conceptual-x-s-update} of $\augconex$ method}\label{alg:implicit}
	\label{alg:x-s-update}
	\begin{algorithmic}
		\State $w_0 = \wh{x}_k$
		\For{ $t = 0, \dots, T-1$}
		\State $w_{t+1} = F(w_t)$
		\EndFor\\
		\Return $w_T$.
	\end{algorithmic}
\end{algorithm}
Algorithm \ref{alg:x-s-update} computes the fixed point of $F$. In what follows, we show that $x_{k+1}$ is a fixed point of $F$. Moreover, if $L_k$ satisfies a simple condition, then $x_{k+1}$ is a unique fixed point and Algorithm \ref{alg:x-s-update} converges linearly. Lemma \ref{lem:F-contraction} states that $F$ is a contraction operator for sufficiently large $L_k$. The proofs and analysis regarding the results in Lemma \ref{lem:F-contraction} and Theorem \ref{thm:linear-convergence} is discussed in Appendix Section \ref{sec:implicit}. 
\begin{lemma}\label{lem:F-contraction}
	If $L_k\geq 2\rho_k(M_g + M_{\chi})^2$, then $F$ is a contraction operator satisfying satisfying $\gnorm{F(\bx_1) - F(\bx_2)}{}{} \le \tfrac{1}{2} \gnorm{\bx_1-\bx_2}{}{}$. Moreover, $x_{k+1}$ is the unique fixed point of $F$.
\end{lemma}

One important consequence of \ref{lem:F-contraction} is that Algorithm \ref{alg:x-s-update} converges to $x_{k+1}$ linearly. Theorem \ref{thm:linear-convergence} formally states such a claim. 
\begin{theorem}\label{thm:linear-convergence}
	Let $\{w_t\}_{t=1}^T$ be the sequence generated by Algorithm \ref{alg:x-s-update}. If conditions in Lemma \ref{lem:F-contraction} are satisfied then 
	$\gnorm {w_T-x_{k+1}}{}{} \le2^{-T}\gnorm{\wh{x}_k - x_{k+1}}{}{}.$ Furthermore, implementing Algorithm \ref{alg:x-s-update} requires only $T$ steps.
\end{theorem}

\begin{remark}
	Note that Algorithm \ref{alg:x-s-update} does not need any new computation of $\mathfrak{F}, g$ or $g'$. Indeed, the linearization of $g$ in the operator $F$ implies that the variable x appears linearly inside the proximal operator. Hence, the gradient complexity of Algorithm \ref{alg:augConEx} exactly equals $K$ even though it requires additional proximal operations in the inner loop. Since $X$ is an easy set to project onto, and the algorithm converges linearly, we can get arbitrarily close to $x_{k+1}$ quickly with little computational effort. Moreover, as the algorithm proceeds, the last iterate $x_k$ will get close to $x^*$. Hence, the distance between $\wh{x}_k$ and $x_{k+1}$ will decrease as well. Hence, there is a warm-starting effect of choosing $w_0 = \wh{x}_k$. 
    The convergence guarantee in Theorem \ref{thm:linear-convergence} implies $\gnorm{w_T- x^{k+1}} \le \delta$ when $T = \log_{2}(\frac{\gnorm{\wh{x}_k - x_{k+1}}{}{}}{\delta})$. However, the base 2 of the log here is an underestimate. E.g., in the convex case, this base is as large as $\sqrt{K}$. In sublinear convergence case, where $K = poly(\frac{1}{\epsilon})$, we have $T = \log_{\sqrt{K}}(\frac{\gnorm{\wh{x}_k - x_{k+1}}{}{}}{\delta}) = O(1)$ for any $\delta = O(\text{poly}(\frac{1}{\epsilon}))$. Hence, the log factor of $T$ vanishes if we carefully analyze the chosen value of $L_k$. 
\end{remark}
We observe that Algorithm \ref{alg:x-s-update} behaves along the expected lines in our numerical experiments. In particular, due to its ``strong'' linear rate, the fixed point scheme converges to $x_{k+1}$ in only four projections in the initial phase of Algorithm \ref{alg:augConEx}. This number quickly comes down to two projections in the middle phase and only one towards the end. Most iterations require two projections (see Section \ref{sec:numerical}).
Due to the fast converging and computationally cheap method for update \eqref{eq:conceptual-x-s-update}, we henceforth assume that it can be computed directly.

We are now ready to discuss the convergence results of the $\augconex$ method. In the following section, we present theorems that illustrate the step size policy and convergence results of $\augconex$ for solving stochastic composite nonsmooth function constrained optimization problems with or without strong-convexity assumption. We provide the detailed proofs of these theorems in Section \ref{sec:convegence-analysis-augconex}. 
\subsection{Convergence rate results for \augconex~Method}
We first discuss the convergence guarantees for the convex case. Recall that in this case, the strong convexity parameter $\mu_f = 0$. 
\begin{theorem}\label{thm:convex-case}
	Suppose \eqref{eq:SO_F_objective}, \eqref{eq:smooth-nonsmooth}, \eqref{eq:sonstraint-smooth-nonsmooth}, \eqref{eq:M_g} are satisfied for the convex problem in \eqref{eq:prima-problem} with $\mu_f = 0$. Let $B\geq 1$ be a constant. Set $K\geq 2$ and assume $\{\rho_{k},\tau_k,L_k,\beta_{k+1}\}_{k\geq 1}$ are generated by the following step size policy parameter in Algorithm \ref{alg:augConEx}
	\begin{equation}\label{eq:non-ergo-stepsize-non-smooth-eta}
		\begin{aligned}
			&\tau_k = \tfrac{2}{k+1}, \quad \rho_{k+1} = \rho_1 (k+2),\quad
			\eta_k = \rho_{1}\tfrac{k^2}{K},\quad
			\beta_{k+1} = \tfrac{(1-\tau_{k})\tau_{k+1}}{\tau_{k}},\\
			& L_k =L=  2(L_f+ BL_g + \rho_{1}K(M_g + M_\chi)^2)+\tfrac{K\sqrt{120K(\mathcal{H}_*^2 + 2(H_f^2+\sigma^2))}}{120D_X},
		\end{aligned}
	\end{equation}
	where $\mathcal{H}_*=L_gD_x  [\gnorm{y^\ast}{}{}+1 - B]_++H_g(\gnorm{y^\ast}{}{}+1) $. Then, we have the following bounds for optimality and feasibility gaps 
	\begin{equation}\label{eq:convex-optim}
		\begin{aligned}
		\Ebb[\psi_0(x_K)- \psi_0(x^*)]\leq&\tfrac{13D_X\sqrt{120K(\mathcal{H}_*^2 + 2(H_f^2+ \sigma^2))}}{60(K-1)}+ \tfrac{1}{(K-1)\rho_{1}}\gnorm{y_1}{}{2} + \tfrac{52D_X^2(L_f+ BL_g +\rho_{1} K(M_g + M_\chi)^2)}{(K-1)K}\\
		& \quad+\tfrac{80D_X(K+3)(H_f^2+ \sigma^2)}{K(240D_X\rho_{1}(M_g+M_\chi)^2 + \sqrt{120K(\mathcal{H}_*^2 + 2(H_f^2+\sigma^2))})}+ \tfrac{2(H_f^2 + \sigma^2)}{K^2(K-1)\rho_1(M_g + M_\chi)^2}.
	\end{aligned}
	\end{equation}
	and 
	\begin{equation}\label{eq:convex-feasib}
		\begin{aligned}
			\Ebb(\gnorm{[\psi(x_K)]_+}{}{})& \leq  \tfrac{13D_X\sqrt{120K(\mathcal{H}_*^2 + 2(H_f^2+ \sigma^2))}}{6(K-1)} + \tfrac{52D_X^2(L_f+ BL_g + \rho_{1}K(M_g + M_\chi)^2)}{(K-1)K}\\
		& \quad + \tfrac{40D_X(K+3)(\mathcal{H}_*^2+2(H_f^2+ \sigma^2) )}{K(240D_X\rho_{1}(M_g+M_\chi)^2 + \sqrt{120K(\mathcal{H}_*^2 + 2(H_f^2+\sigma^2))})}+ \tfrac{1}{(K-1)\rho_{1}}\gnorm{y_1 - \hat{y}}{}{2} \\
		&\quad +\tfrac{\mathcal{H}_*^2+2(H_f^2 + \sigma^2)}{K^2(K-1)\rho_1(M_g + M_\chi)^2},
		\end{aligned}
	\end{equation} 
 where $\hat{y} = (\gnorm{y^\ast}{}{}+1)\tfrac{[\psi(\bar{x}_K)]_+}{\gnorm{[\psi(\bar{x}_K)]_+}{}{}}$.
\end{theorem}
one immediate result of Theorem \ref{thm:convex-case} is the following corollary
\begin{corollary}\label{cor:convex}
	we obtain an $(\epsilon,\epsilon)$-optimal solution of the problem \eqref{eq:prima-problem} in $K_\epsilon$ iterations, where 
	\begin{equation}\label{eq:K-epsilon-convex}
		\begin{aligned}
			K_\epsilon =& \max \big\{  \tfrac{(130D_x)^2(\mathcal{H}_*^2 +2( H_f^2+\sigma^2))}{0.3\epsilon^2} + \tfrac{(400D_x)^2(\mathcal{H}_*^2 +2( H_f^2+\sigma^2))}{120\epsilon^2}+1,\tfrac{10}{\epsilon\rho_{1}}\gnorm{y_1 - \hat{y}}{}{2}+\tfrac{520D_x^2\rho_1(M_g + M_\chi)^2}{\epsilon}+2,\\
			&(\tfrac{5(H_f^2 + \sigma^2)}{\epsilon\rho_1(M_g + M_\chi)^2})^{1/3}+1, (\tfrac{260D_x^2(L_f+BL_g)}{\epsilon})^{1/2}+1, (\tfrac{3000D_x^2(\mathcal{H}_*^2 +8( H_f^2+\sigma^2))}{\epsilon^2})^{1/3}.
			\big\}
		\end{aligned}
	\end{equation}
\end{corollary}
\begin{proof}
	Using \eqref{eq:convex-feasib}, and \eqref{eq:K-epsilon-convex}, we have $	\Ebb[\gnorm{[\psi(x_K)]_+}{}{}] \leq \tfrac{\epsilon}{5}+\tfrac{\epsilon}{5}+\tfrac{\epsilon}{5}+\tfrac{\epsilon}{5}+\tfrac{\epsilon}{5}$. Similarly, using \eqref{eq:convex-optim} and \eqref{eq:K-epsilon-convex}, one can see $ \Ebb[\psi_0(x_K) - \psi(x^*)] \leq \epsilon$.
\end{proof}
Theorem \ref{thm:convex-case} states the convergence result of the $\augconex$ method for the stochastic nonsmooth case of the optimization problem in \eqref{eq:prima-problem}. One important observation is related to the choice of the initial parameter $\rho_{1}$. Specifically, to have the error complexity of $\mathcal{O}(1/\epsilon^2)$ or the convergence rate of $ \mathcal{O}(1/\sqrt{K}) $, one should choose $\rho_{1}$ as $\mathcal{O}(\sqrt{K})$ and $\Omega(1/\sqrt{K})$. In detail, one can derive the following error complexity results for nonsmooth deterministic and stochastic settings. First, let us derive the complexity for the nonsmooth deterministic case (i.e., $\sigma=0$). Considering $\rho_{1}=1$ the error complexity after $K$ iterations is of 
\[
\mathcal{O}\big( \tfrac{D_X^2(\mathcal{H}_*^2 + H_f^2)}{\epsilon^2} + \tfrac{D_X\sqrt{(L_f+BL_g)}}{\sqrt{\epsilon}} + \tfrac{D_X^2(M_g+M_\chi)^2}{\epsilon}\big).
\]
Let us consider the stochastic case where $\sigma>0$. We have the following error complexity for $\rho_{1} =1$
\[
\mathcal{O}\big( \tfrac{D_X^2(\mathcal{H}_*^2 + H_f^2+\sigma^2)}{\epsilon^2} + \tfrac{D_X\sqrt{(L_f+BL_g)}}{\sqrt{\epsilon}} + \tfrac{D_X^2(M_g+M_\chi)^2}{\epsilon}\big).
\]
Observe that the effect of having a stochastic approximation of $f'$ is the same as non smoothness of $f$. 

Now, we specify a step size policy for the convergence of $\augconex$ method for solving problem \eqref{eq:prima-problem} in the strongly convex setting where $f$ is $\mu_f$-strongly convex. The proof of Theorem \ref{thm:non-ergo-strong-semi} is presented in Section~\ref{sec:convegence-analysis-augconex}. 
\begin{theorem}\label{thm:non-ergo-strong-semi}
	Suppose \eqref{eq:SO_F_objective}, \eqref{eq:smooth-nonsmooth}, \eqref{eq:sonstraint-smooth-nonsmooth}, \eqref{eq:M_g} are satisfied, $\mu_f>0$ and we have the following step size policy for $\{\tau_{k},\rho_{k},\eta_k,\beta_{k+1}\}_{k\geq 1}$
	\begin{equation}\label{eq:non-ergo-stepsize-strong}
		\begin{aligned}
			&\tau_{k+1} = \tfrac{\tau_k}{2}(\sqrt{\tau_k^2+4}-\tau_k), \quad \rho_{k+1} = \tfrac{\rho_1}{\tau_{k+1}^2},\quad
			\eta_k = \rho_k,\quad \beta_{k+1} = \tfrac{(1-\tau_{k})\tau_{k} L_{k}}{\tau_{k}^2 L_{k} + L_{k+1}\tau_{k+1}}\\
			& L_k =2(L_f + B L_g+\rho_k (M_g + M_\chi)^2),
		\end{aligned}
	\end{equation}
	where $\tau_1 = 1$, $ \rho_{1}=\tfrac{\mu_f}{2(M_g + M_\chi)^2}$. 
	Then we have the following expected optimality and feasibility gaps
	\begin{equation}\label{eq:opt-1-non-ergo}
		\begin{aligned}
			\Ebb[\psi_0(x_K) - \psi(x^*)]\leq& \tfrac{4}{K^2}[4(L_f+BL_g)D_X^2+ \tfrac{1}{2\eta_{1}}\gnorm{y_1}{}{2} ]+ \tfrac{8(K+2)( H_f^2+\sigma^2)}{\mu_fK^2}
		\end{aligned}
	\end{equation}
	\begin{equation}\label{eq:feasib-1-non-ergo}
		\begin{aligned}
			\Ebb[\gnorm{[\psi(x_K)]_+}{}{}]&\leq \tfrac{4}{K^2}[(L_f+BL_g)\gnorm{x_1-x^*}{}{2}
			+ \tfrac{1}{2\eta_{1}}\gnorm{y_1-\hat{y}}{}{2} ]+ \tfrac{4(K+2)(\mathcal{H}_*^2+ 2(H_f^2+\sigma^2))}{\mu_fK^2},
		\end{aligned}
	\end{equation}
	where $\hat{y} = (\gnorm{y^*}{}{}+1)\tfrac{[\psi(\bar{x}_K)]_+}{\gnorm{[\psi(\bar{x}_K)]_+}{}{}}$ and $\mathcal{H}_* = L_gD_x[\gnorm{y^*}{}{}+1-B]_++ H_g(\gnorm{y^*}{}{}+1)$.

 Also, convergence in the last iterate can be stated as below 
 \begin{equation}\label{eq:last-iter}
     \begin{aligned}
	\Ebb[\gnorm{\tfrac{1}{\tau_{K-1}}[x_{K} - (1-\tau_{K-1})x_{K-1}] - x^*}{}{2}]&\leq \tfrac{8}{\mu_f K^2}[4(L_f+BL_g)\gnorm{x_1-x^*}{}{2}
	+ \tfrac{2}{\eta_{1}}\gnorm{y_1-y^*}{}{2}\\
        &\quad+ \tfrac{4(K+2)((L_gD_x[\gnorm{y^*}{}{}-B]_++ H_g\gnorm{y^*}{}{})^2+ 2(H_f^2+\sigma^2))}{\mu_f}],
	\end{aligned}
 \end{equation}
\end{theorem}
Similar to Corollary \ref{cor:convex}, one can derive the following corollary from the above theorem. 
\begin{corollary}\label{cor:strong-convex}
	we obtain an $(\epsilon,\epsilon)$-optimal solution of the problem \eqref{eq:prima-problem} in $K_\epsilon$ iterations, where 
	\begin{equation}\label{eq:K-epsilon}
		K_\epsilon = \max \big\{ \tfrac{8D_x\sqrt{L_f+BL_g}}{\sqrt{\epsilon}} + \tfrac{2\sqrt{2}\gnorm{y_1-y^*}{}{}}{\sqrt{\eta_1\epsilon}} +\tfrac{4\sqrt{2(\mathcal{H}_*^2 +2( H_f^2+\sigma^2))}}{\sqrt{\mu_f\epsilon}} , \tfrac{16(\mathcal{H}_*^2 +2( H_f^2+\sigma^2))}{\mu_f\epsilon}\}.
	\end{equation}
\end{corollary}
\begin{proof}
	Using \eqref{eq:feasib-1-non-ergo}, and \eqref{eq:K-epsilon}, we have $	\Ebb[\|[\psi(x_K)]_+\|] \leq \tfrac{\epsilon}{4}+\tfrac{\epsilon}{4}+\tfrac{\epsilon}{4}+\tfrac{\epsilon}{4}$. Similarly, using \eqref{eq:opt-1-non-ergo} and \eqref{eq:K-epsilon}, one can see $ \Ebb[\psi_0(x_K) - \psi(x^*)] \leq \epsilon$.
\end{proof}
Theorem \ref{thm:non-ergo-strong-semi} and Corollary \ref{cor:strong-convex} provide iteration complexity for the stochastic nonsmooth functional constrained optimization problem with strong convexity assumption. In general $ \augconex $ obtain the following convergence rate for the strongly convex problem 
\[
\mathcal{O}\big( \tfrac{D_x\sqrt{L_f+BL_g}}{\sqrt{\epsilon}}+ \tfrac{\mathcal{H}_*^2+H_f^2+\sigma^2}{\mu_f\epsilon}\big).
\] 
Considering convergence rate and choice of $\rho_{1}$, the overall convergence rate is of order $\mathcal{O}(\tfrac{1}{\rho_{1}K})$ meaning that choosing an extremely small value for $\rho_{1}$ will lead to divergence. Specifically, $\rho_{1}$ should be in order of $\Omega(K^{-1})$. Furthermore, one can observe that the change rate in $L_k$ is of order $\mathcal{O}(k^2)$. One important aspect of the convergence rate of $\augconex$ method is its low dependency on the Lipschitz constant $L_f$ as it comes with factor $\mathcal{O}(1/\sqrt{\epsilon})$. Also, the significance of $L_g$ is controlled by choice of $B$, i.e., $B\geq \|y^*\|+1$. 

 \section{Convergence Analysis of $ \augconex $ Method}\label{sec:convegence-analysis-augconex}
In this section, we provide a comprehensive analysis of Theorem \ref{thm:convex-case} and Theorem \ref{thm:non-ergo-strong-semi}. First, since $\augconex$ is a primal-dual algorithm, one can use the primal-dual gap as a measure of convergence. In particular, for two feasible solutions $z = (x,s,y)$ and $\bar{z} = (\bar{x},\bar{s},\bar{y})$, one can define the primal-dual gap as follows 
\[
Gap:= \Lcal(\bar{x},\bar{s},y) - \Lcal(x,s,\bar{y}).
\]
In the following key lemma, we construct a recursive relation in terms of the iterates $\{x_{k+1},s_{k+1},\tilde{y}_{k+1},\bar{y}_{k+1}\}_{k\geq 1}$ generated by Algorithm \ref{alg:augConEx}. Before formally stating the lemma,
notice from optimality of $s_{k+1}$ and the definition of $y_{k+1}$ based on \eqref{eq:Denotes} of Algorithm \ref{alg:augConEx}, one can say $y_{k+1}= -\grad h(s_{k+1})$ where
\[
h(s): = \inprod{-\wt{y}_k}{s} + \tfrac{\rho_k}{2}\gnorm{\ell_g(x_{k+1})+ \chi(x_{k+1}) - (1-\tau_k)V_k -s}{}{2}.
\]
Moreover, since $\grad h(s_{k+1})  + N_{\{s:s\leq \mathbf{0}\}}(s_{k+1})\ni \mathbf{0} $, one can conclude that $y_{k+1}\in  N_{\{s:s\leq \mathbf{0}\}}(s_{k+1})$ meaning $y_{k+1}\geq \mathbf{0}$. Now, we are ready to mention the Lemma \ref{lemma:importantlemma}. 
\begin{lemma}\label{lemma:importantlemma}
	Suppose $\tfrac{\eta_k}{2} \le \rho_k$ and $L_k\geq 2L_f+2BL_g$, where $B\geq 0$. Then for all $(x,s,y)\in X\times \mathbb{R}_{-}^m\times\mathbb{R}_{+}^m$ and $k=1,\dots, K-1$, we have
	\begin{align*}
		\Lcal(x_{k+1}, &s_{k+1}, y) - \Lcal(x,s,\wb{y}_{k+1}) - (1-\tau_k)[\Lcal(x_k, s_k, y) - \Lcal(x,s,\wb{y}_k)] \\
		&\le\tfrac{(L_k-\mu_f)\tau_k^2}{2}\gnorm{\tfrac{1}{\tau_k}[\wh{x}_k-(1-\tau_k)x_k] -x}{}{2} - \tfrac{L_k\tau_k^2}{2}\gnorm{\tfrac{1}{\tau_k}[x_{k+1} - (1-\tau_k)x_k] - x}{}{2} -\tau_k\inner{\delta_k }{\dot{x}_k-x} \nonumber\\
		&\quad + \tfrac{1}{2\eta_k}[\gnorm{y-\wt{y}_k}{}{2} - \gnorm{y-\wt{y}_{k+1}}{}{2} ] + \tfrac{(L_gD_x[\gnorm{y}{}{}-B]_++ H_g\gnorm{y}{}{})^2}{L_k-2(L_f + BL_g)}+ \tfrac{2(H_f^2+\gnorm{\delta_k}{}{2})}{L_k}- \tfrac{\mu_f\tau_k(1-\tau_k)}{2}\gnorm{x_k-x}{}{2}.  
	\end{align*}
	where $\ell_g(x_{k+1}) $ is defined as \eqref{eq:ell_g}, $\delta_k =  \mathfrak{F}(\hat{x}_k,\xi_k)-f'(\hat{x}_k)$ and $\dot{x}_k = \tfrac{\hat{x}_k-(1-\tau_{k})x_k}{\tau_k}$. 
\end{lemma}
\begin{proof}
	First, using optimality of $x_{k+1}$, we have for all $x\in X$
	\begin{equation}\label{eq:int_rel2}
		\begin{aligned}
			\inprod{f'(\wh{x}_k) + g'(\wh{x}_k)y_{k+1}}{x_{k+1}-x}&+ \inner{\delta_k }{\hat{x}_k-x} + \chi_0(x_{k+1}) - \chi_0(x) + \inner{\chi(x_{k+1})-\chi(x))}{y_{k+1}} \\
			& \le \tfrac{L_k}{2} [\gnorm{x-\wh{x}_k}{}{2} - \gnorm{x-x_{k+1}}{}{2} - \gnorm{x_{k+1}-\wh{x}_k}{}{2}]+\inner{\delta_k}{\hat{x}_k-x_{k+1}}.
		\end{aligned}
	\end{equation}
	Using strong-convexity and convexity of $f, g$ respectively, Lipschitz property of $f$ and definition of $\ell_g(x_{k+1})$, we have for all $x\in X$
	\begin{align*}
		f(x_{k+1}) &\le f(\wh{x}_k) + \inprod{f'(\wh{x}_k)}{x_{k+1}-\wh{x}_k} + \tfrac{L_f}{2}\gnorm{x_{k+1}-\wh{x}_k}{}{2} + H_f\gnorm{x_{k+1}-\wh{x}_k}{}{},\\
		f(\wh{x}_k) &\le f(x) + \inprod{f'(\wh{x}_k)}{\wh{x}_k-x}- \tfrac{\mu_f\gnorm{x-\hat{x}_k}{}{2}}{2},\\
		g(\wh{x}_k) &\le g(x) + \inprod{g'(\wh{x}_k)}{\wh{x}_k-x},\\
		\ell_g(x_{k+1})&= g(\wh{x}_k) + \inprod{g'(\wh{x}_k)}{x_{k+1}-\wh{x}_k}.
	\end{align*}
	Multiplying the last two relations by $y_{k+1}$, noting that $y_{k+1} \ge \mathbf{0}$ and then summing the above relations, we have for all $x \in X$
	\begin{align*}
		f(x_{k+1}) -f(x) &+ \inprod{y_{k+1}}{\ell_g(x_{k+1}) - g(x)}-  \tfrac{L_f}{2}\gnorm{x_{k+1}-\wh{x}_k}{}{2} - H_f\gnorm{x_{k+1}-\wh{x}_k}{}{} + \tfrac{\mu_f\gnorm{x-\hat{x}_k}{}{2}}{2} \\
		&\le \inprod{f'(\wh{x}_k) + g'(\wh{x}_k)y_{k+1}}{x_{k+1}-x}. 
	\end{align*}
	Using the above relation inside \eqref{eq:int_rel2}, we have for all $x\in X$
	\begin{equation}\label{eq:int_rel4}
		\begin{split}
			f(x_{k+1}) -& f(x) + \inprod{y_{k+1}}{\ell_g(x_{k+1}) -g(x)}+\inner{\delta_k }{\hat{x}_k-x} +\chi_0(x_{k+1}) - \chi_0(x) + \inner{\chi(x_{k+1})-\chi(x))}{y_{k+1}}\\ &\le  \tfrac{(L_k-\mu_f)\gnorm{x-\wh{x}_k}{}{2}}{2}- \tfrac{L_k\gnorm{x-x_{k+1}}{}{2}}{2}
			- \tfrac{(L_k-L_f)\gnorm{x_{k+1}-\wh{x}_k}{}{2}}{2} + H_f\gnorm{x_{k+1}-\wh{x}_k}{}{}+\inner{\delta_k}{\hat{x}_k-x_{k+1}}.
		\end{split}
	\end{equation}
	Similarly, using first-order optimality condition of $s_{k+1}$ and noting $y_{k+1}\geq \mathbf{0}$, we have for all $s \le \mathbf{0}$, 
	\begin{equation}\label{eq:int_rel3}
		\inprod{y_{k+1}}{s - s_{k+1}} \le 0.
	\end{equation}
	Also, let us consider these two equations 
	\begin{equation}\label{eq:recurtion}
		\begin{aligned}
			&(1-\tau_k)\gnorm{x_{k+1} - x_{k}}{}{2} + \tau_k\gnorm{x_{k+1}-x}{}{2} = \gnorm{x_{k+1}-(1-\tau_k)x_k-\tau_kx}{}{2} + \tau_k(1-\tau_k)\gnorm{x_k-x}{}{2},\\
			&(1-\tau_k)\gnorm{\hat{x}_k - x_{k}}{}{2} + \tau_k\gnorm{\hat{x}_k-x}{}{2} = \gnorm{\hat{x}_k-(1-\tau_k)x_k-\tau_kx}{}{2} + \tau_k(1-\tau_k)\gnorm{x_k-x}{}{2},\quad \forall x\in X.
		\end{aligned}
	\end{equation}
	Now, we sum the following four relations: (1) Multiply \eqref{eq:int_rel4} by $\tau_k \in (0,1]$, (2) multiply \eqref{eq:int_rel3} by $\tau_k$, (3) use $x=x_k$ in \eqref{eq:int_rel4} and multiply by $1-\tau_k$, and (4) use $s= s_k$ in \eqref{eq:int_rel3} and multiply by $1-\tau_k$. Then regarding \eqref{eq:recurtion}, we have for all $x \in X, s \le \mathbf{0}$
	\begin{align}
		&f(x_{k+1}) + \chi_0(x_{k+1}) - \tau_k(f(x)+ \chi_0(x)) -(1-\tau_k)(f(x_k)+ \chi_0(x_k))+\inner{\delta_k }{\hat{x}_k-\tau_kx - (1-\tau_k)x_k}\label{eq:int_rel5}\\
		&\quad  + \inprod{y_{k+1}}{\ell_g(x_{k+1})+ \chi(x_{k+1}) - s_{k+1} - \tau_k (g(x)+\chi(x)-s) - (1-\tau_k)(g(x_k)+\chi(x_k)-s_k)}\nonumber\\
		& \le  \tfrac{(L_k-\mu_f)\tau_k^2}{2}\gnorm{\tfrac{1}{\tau_k}[\wh{x}_k-(1-\tau_k)x_k] -x}{}{2} - \tfrac{L_k\tau_k^2}{2}\gnorm{\tfrac{1}{\tau_k}[x_{k+1} - (1-\tau_k)x_k] - x}{}{2}\nonumber\\
		&\quad - \tfrac{\mu_f\tau_k(1-\tau_k)}{2}\gnorm{x_k-x}{}{2}- \tfrac{L_k-L_f}{2}\gnorm{x_{k+1} - \wh{x}_k}{}{2} +  H_f\gnorm{x_{k+1}-\wh{x}_k}{}{}+\inner{\delta_k}{\hat{x}_k-x_{k+1}}\nonumber.
	\end{align}
	Finally, noting that $\wt{y}_{k+1} = \argmin_{y} -\inprod{\wt{V}_{k+1} - (1-\tau_k)V_k}{y} + \tfrac{1}{2\eta_k} \gnorm{y-\wt{y}_k}{}{2}$, we have
	\begin{align*}
		\inprod{\wt{V}_{k+1} - (1-\tau_k)V_k}{y-\wt{y}_{k+1}} \le \tfrac{1}{2\eta_k}[\gnorm{y-\wt{y}_k}{}{2} - \gnorm{y-\wt{y}_{k+1}}{}{2} - \gnorm{\wt{y}_{k+1}-\wt{y}_k}{}{2}].
	\end{align*}
	Adding the above relation to \eqref{eq:int_rel5} and noting definitions of $\wt{V}_{k+1}, V_k$ with respect to \eqref{eq:Denotes},  we have
	\begin{equation}\label{eq:int_rel6}
		\begin{aligned}
			&f(x_{k+1}) + \chi_0(x_{k+1}) - \tau_k(f(x)+ \chi_0(x)) -(1-\tau_k)(f(x_k)+ \chi_0(x_k))+\inner{\delta_k }{\hat{x}_k-\tau_kx - (1-\tau_k)x_k}\\
			&\quad-\inprod{\tau_k y_{k+1}}{g(x)+\chi(x)-s} + \inprod{y}{\wt{V}_{k+1} - (1-\tau_k)V_k}\\
			& \le \inprod{\wt{y}_{k+1}-y_{k+1}}{\wt{V}_{k+1} - (1-\tau_k)V_k} - \tfrac{L_k-L_f}{2}\gnorm{x_{k+1} - \wh{x}_k}{}{2} +  H_f\gnorm{x_{k+1}-\wh{x}_k}{}{}\\
			&\quad +\tfrac{(L_k-\mu_f)\tau_k^2}{2}\gnorm{\tfrac{1}{\tau_k}[\wh{x}_k-(1-\tau_k)x_k] -x}{}{2} - \tfrac{L_k\tau_k^2}{2}\gnorm{\tfrac{1}{\tau_k}[x_{k+1} - (1-\tau_k)x_k] - x}{}{2}\\
			&\quad + \tfrac{1}{2\eta_k}[\gnorm{y-\wt{y}_k}{}{2} - \gnorm{y-\wt{y}_{k+1}}{}{2} - \gnorm{\wt{y}_{k+1}-\wt{y}_k}{}{2}]+\inner{\delta_k }{\hat{x}_k-x_{k+1}}- \tfrac{\mu_f\tau_k(1-\tau_k)}{2}\gnorm{x_k-x}{}{2}.
		\end{aligned}
	\end{equation}
	From the definitions of $y_{k+1}$ and $\wt{y}_{k+1}$ in \eqref{eq:conceptual-x-s-update} and \eqref{eq:Denotes} of Algorithm \ref{alg:augConEx}, we have 
	$\inprod{\wt{y}_{k+1}-y_{k+1}}{\wt{V}_{k+1} - (1-\tau_k)V_k} = (\eta_k-\rho_k) \gnorm{\wt{V}_{k+1} - (1-\tau_k)V_k}{}{2}$. Moreover, $-\tfrac{1}{2\eta_k}\gnorm{\wt{y}_{k+1} - \wt{y}_k}{}{2} = -\tfrac{\eta_k}{2}\gnorm{\wt{V}_{k+1} - (1-\tau_k)V_k}{}{2}$. Using these two relations in \eqref{eq:int_rel6}, alongside with \eqref{eq:sp-problem-conj} and Cauchy-Schwarz inequality for $\inner{\delta_k }{\hat{x}_k-x_{k+1}}$ we have 
	\begin{align*}
		&\Lcal(x_{k+1}, s_{k+1}, y) - \Lcal(x,s,\wb{y}_{k+1}) - (1-\tau_k)[\Lcal(x_k, s_k, y) - \Lcal(x,s,\wb{y}_k)] \\
		& \le \tfrac{\eta_k - 2\rho_k}{2}\gnorm{\wt{V}_{k+1} - (1-\tau_k)V_k}{}{2} - \tfrac{L_k-L_f}{2}\gnorm{x_{k+1} - \wh{x}_k}{}{2} +  (H_f+\|\delta_k\|)\gnorm{x_{k+1}-\wh{x}_k}{}{}\nonumber\\
		&\quad +\tfrac{(L_k-\mu_f)\tau_k^2}{2}\gnorm{\tfrac{1}{\tau_k}[\wh{x}_k-(1-\tau_k)x_k] -x}{}{2} - \tfrac{L_k\tau_k^2}{2}\gnorm{\tfrac{1}{\tau_k}[x_{k+1} - (1-\tau_k)x_k] - x}{}{2}- \tfrac{\mu_f\tau_k(1-\tau_k)}{2}\gnorm{x_k-x}{}{2} \nonumber\\
		&\quad + \tfrac{1}{2\eta_k}[\gnorm{y-\wt{y}_k}{}{2} - \gnorm{y-\wt{y}_{k+1}}{}{2} ] + \inprod{y}{g(x_{k+1}) - \ell_g(x_{k+1})}-\inner{\delta_k }{\hat{x}_k-\tau_kx - (1-\tau_k)x_k}.
	\end{align*}
	Considering the following relations, one can conclude the proof
	\begin{subequations}
		\begin{equation*}
			(H_f+\|\delta_k\|)\gnorm{x_{k+1}-\wh{x}_k}{}{}\nonumber \leq \tfrac{L_k\gnorm{x_{k+1}-\wh{x}_k}{}{2}\nonumber}{4} + \tfrac{2(H_f^2+\gnorm{\delta_k}{}{2})}{L_k},
		\end{equation*}
		\begin{equation*}
			\begin{aligned}
				\langle y , g(x_{k+1})- \ell_g(x_{k+1})\rangle \leq& \tfrac{L_g}{2}  \gnorm{y}{}{}\gnorm{x_{k+1}- \hat{x}_k}{}{2} + H_g\gnorm{y}{}{}\gnorm{x_{k+1}- \hat{x}_k}{}{}, \\
				& = \tfrac{L_g}{2} (\gnorm{y}{}{} - B)\gnorm{x_{k+1}- \hat{x}_k}{}{2} + \tfrac{B L_g}{2}\gnorm{x_{k+1}- \hat{x}_k}{}{2} + H_g\gnorm{y}{}{}\gnorm{x_{k+1}- \hat{x}_k}{}{},\\
				& \leq (L_gD_x [\gnorm{y}{}{} - B]_++H_g\gnorm{y}{}{}) \gnorm{x_{k+1}- \hat{x}_k}{}{} + \tfrac{B L_g}{2}\gnorm{x_{k+1}- \hat{x}_k}{}{2},
			\end{aligned}
		\end{equation*}
		\begin{equation*}
			(L_gD_x  [\gnorm{y}{}{} - B]_++H_g\gnorm{y}{}{})\gnorm{x_{k+1}- \hat{x}_k}{}{} - \tfrac{L_k-2(L_f + BL_g)}{4} \gnorm{x_{k+1}- \hat{x}_k}{}{2}\leq \tfrac{(L_gD_x  [\gnorm{y}{}{} - B]_++H_g\gnorm{y}{}{})^2}{L_k-2(L_f + BL_g)}.
		\end{equation*}
	\end{subequations}
\end{proof}
Lemma \ref{lemma:importantlemma} is crucial to obtain optimality and feasibility gaps in Definition \ref{def:gaps}. Indeed, we use Lemma \ref{lemma:importantlemma} to construct convergence results for convex and strongly convex settings of the optimization problem in \ref{eq:prima-problem}. In the following section, we provide the detailed technical results we used to obtain Theorem \ref{thm:convex-case}. 
\subsection{Convergence analysis for the convex problem}
\paragraph{Proof of Theorem \ref{thm:convex-case}}
First, assume the following conditions for the convex problem
\begin{equation}\label{eq:conditions}
	\tfrac{\eta_k}{2}\leq \rho_k, \quad L_k\geq 2(L_f+ BL_g + \rho_k(M_g+M_\chi)^2),\quad \tfrac{1}{\eta_k}\leq \tfrac{1-\tau_k}{\eta_{k-1}}.
\end{equation}
Thus, one can easily verify that the step-size policy in \eqref{eq:non-ergo-stepsize-non-smooth-eta} satisfies the conditions in \eqref{eq:conditions} for $k\geq 2$. 
In the view of \eqref{eq:Denotes} and \eqref{eq:non-ergo-stepsize-non-smooth-eta}, we let $\tilde{x}_k  = \tfrac{1}{\tau_k}[\hat{x}_k-(1-\tau_k)x_k]$ and $\tilde{x}_{k+1}  = \tfrac{1}{\tau_k}[x_{k+1}-(1-\tau_k)x_k]$. Taking the expectation from both sides regarding \eqref{eq:SO_F_objective}, we can write the following recursion for $k\geq 2$
\begin{equation}\label{eq:non-ergo-5-non-smooth-recur-eta}
	\begin{aligned}
		&\Ebb[\Lcal(x_{k+1},s_{k+1}, y)- \Lcal(x,s,\bar{y}_{k+1})  + \tfrac{1}{2\eta_k}\gnorm{\tilde{y}_{k+1}-y}{}{2}]\leq \Ebb\big[
		(1-\tau_k)[ \Lcal(x_{k},s_{k}, y)- \Lcal(x,s,\bar{y}_{k})
		+  \tfrac{1}{2\eta_{k-1}}\gnorm{\tilde{y}_k-y}{}{2}\big]\\
		&+\Ebb\big[\tfrac{\tau_k^2L_k}{2}[\gnorm{\tilde{x}_k-x}{}{2}- \gnorm{\tilde{x}_{k+1}-x}{}{2}]\big] -\tau_k\Ebb[\inner{\delta_k }{\dot{x}_k-x}]
		+ \Ebb\big[ \tfrac{(L_gD_x[\gnorm{y}{}{}-B]_++ H_g\gnorm{y}{}{})^2}{L_k-2(L_f + BL_g)}\big]+ \tfrac{2(H_f^2+ \sigma^2)}{L_k}.
	\end{aligned}
\end{equation}
Consequently, one can write the following recursion 
\begin{equation}\label{eq:recur-nonsmooth-eta}
	\begin{aligned}
		&\Ebb[\Lcal(x_{K},s_{K}, y)- \Lcal(x,s,\bar{y}_{K})] \leq  [\prod_{k=2}^{K-1}(1-\tau_k)]\Ebb[\Lcal(x_{2},s_{2}, y)- \Lcal(x,s,\bar{y}_{2})+ \tfrac{1}{2\eta_{1}}\gnorm{\tilde{y}_2-y}{}{2}]\\
		& \quad  +\Ebb\big[ \sum_{k=2}^{K-1} [\prod_{i=k+1}^{K-1} (1-\tau_{i})] \tfrac{\tau_k^2L_k}{2}[\gnorm{\tilde{x}_k-x}{}{2}-\gnorm{\tilde{x}_{k+1}-x}{}{2}]\big] -\Ebb\big[\sum_{k=2}^{K-1}[\prod_{i=k+1}^{K-1}(1-\tau_{i})]\tau_k\inner{\delta_k}{\dot{x}_k-x}\big] \\
		& \quad + \sum_{k=2}^{K-1}\big[[ \prod_{i=k+1}^{K-1} (1-\tau_{i})] \Ebb\big[ \tfrac{(L_gD_x[\gnorm{y}{}{}-B]_++ H_g\gnorm{y}{}{})^2}{L_k-2(L_f + BL_g)}\big] + \tfrac{2(H_f^2+ \sigma^2)}{L_k}\big].
	\end{aligned}
\end{equation}
Now, let us simplify each line in the RHS of \eqref{eq:recur-nonsmooth-eta}. Note that from Lemma \ref{lemma:importantlemma} and $L_1 = L, \eta_{1} = \tfrac{\rho_{1}}{K}, \hat{x}_1 = x_1 , \tilde{y}_1 = y_1, \tau_1=1$ , we have 
\begin{equation*}
	\begin{aligned}
		\Ebb&[\Lcal(x_{2},s_{2}, y)-\Lcal(x,s,\bar{y}_{2})+ \tfrac{1}{2\eta_{1}}\gnorm{\tilde{y}_2-y}{}{2}]\leq (1-\tau_1)\Ebb[(\Lcal(x_{1},s_{1}, y)- \Lcal(x,s,\bar{y}_{1}))]\\
		& \quad + 2LD_X^2 + \tfrac{K}{2\rho_{1}}\Ebb[\gnorm{y_1 - y}{}{2}] + \Ebb\big[\tfrac{(L_gD_x  [\gnorm{y}{}{} - B]_++H_g\gnorm{y}{}{})^2}{2\rho_1K(M_g + M_\chi)^2}\big] + \tfrac{H_f^2 + \sigma^2}{\rho_1K(M_g + M_\chi)^2} - \Ebb[\inner{\delta_1}{\dot{x}_1-x}].
	\end{aligned}
\end{equation*}
Noting that $\tau_1 = 1$ and $\prod_{k=2}^{K-1}(1-\tau_k) = \tfrac{2}{K(K-1)}$ we have the following relation for the first line in RHS of \eqref{eq:recur-nonsmooth-eta}
\begin{equation}\label{eq:first summation}
	\begin{aligned}
		[\prod_{k=2}^{K-1}(1-\tau_k)]&\Ebb[\Lcal(x_{2},s_{2}, y)- \Lcal(x,s,\bar{y}_{2})+ \tfrac{1}{2\eta_{1}}\gnorm{\tilde{y}_2-y}{}{2}]\leq \tfrac{4LD_X^2}{K(K-1)} + \tfrac{1}{(K-1)\rho_1}\Ebb[\gnorm{y_1 - y}{}{2}]\\
		&\quad +\Ebb\big[\tfrac{(L_gD_x  [\gnorm{y}{}{} - B]_++H_g\gnorm{y}{}{})^2}{K^2(K-1)\rho_1(M_g + M_\chi)^2}\big]+ \tfrac{2(H_f^2 + \sigma^2)}{K^2(K-1)\rho_1(M_g + M_\chi)^2} - \tfrac{2\Ebb[\inner{\delta_1}{\dot{x}_1-x}]}{K(K-1)},
	\end{aligned}
\end{equation}
For the first summation  in the second line of \eqref{eq:recur-nonsmooth-eta} we have 
\begin{equation}\label{eq:second summation}
	\begin{aligned}
		\Ebb\big(\sum_{k=2}^{K-2}& [\prod_{i=k+1}^{K-1} (1-\tau_{i})] \tfrac{\tau_k^2L_k}{2}[\gnorm{\tilde{x}_k-x}{}{2}-\gnorm{\tilde{x}_{k+1}-x}{}{2}]\big) +\Ebb\bracket[\big]{\tfrac{\tau_{K-1}^2L_{K-1}}{2}[\gnorm{\tilde{x}_{K-1}-x}{}{2}-\gnorm{\tilde{x}_{K}-x}{}{2}]}\\
		&  \leq \Ebb\big( \sum_{k=2}^{K-2} \tfrac{ k (k+1)}{(K-1)K} \tfrac{4}{(k+1)^2} \tfrac{ L}{2}[\gnorm{\tilde{x}_k-x}{}{2}-\gnorm{\tilde{x}_{k+1}-x}{}{2}]\big) + \tfrac{8D_X^2L}{K^2} \\
		& = \Ebb\big(\sum_{k=2}^{K-2} \tfrac{2 k }{(K-1)K (k+1)} L [\gnorm{\tilde{x}_k-x}{}{2}-\gnorm{\tilde{x}_{k+1}-x}{}{2}]\big) + \tfrac{8D_X^2L}{K^2}\\
		& = \tfrac{2L}{(K-1)K}\Ebb \Big(\big(\sum_{k=2}^{K-2} \big[\tfrac{k+1}{k+2}
		-\tfrac{k}{k+1}\big] \gnorm{\tilde{x}_{k+1}-x}{}{2} \big)+ \tfrac{2}{3}  \gnorm{\tilde{x}_{2}-x}{}{2}\Big) + \tfrac{8D_X^2L}{K^2} \\
		&  = \tfrac{2L}{(K-1)K} \sum_{k=2}^{K-2} \big(\tfrac{k+1}{k+2} - \tfrac{k}{k+1}\big) \Ebb(\gnorm{\tilde{x}_{k+1}-x}{}{2})  + \tfrac{4L}{3(K-1)K}  \Ebb[\gnorm{\tilde{x}_{k+1}-x}{}{2}] + \tfrac{8D_X^2L}{K^2}\\
		& \leq  \tfrac{8L}{(K-1)K} D_x^2 + \tfrac{16L}{3(K-1)K} D_x^2 + \tfrac{8L}{K^2}D_X^2\leq  \tfrac{22L}{(K-1)K}D_X^2.
	\end{aligned}
\end{equation}
The last summation line in the RHS of \eqref{eq:recur-nonsmooth-eta} has the following upper bound.
\begin{equation}\label{eq:third summation}
	\begin{aligned}
		&\sum_{k=2}^{K-2}\big[[ \prod_{i=k+1}^{K-1} (1-\tau_{i})] \Ebb\big[ \tfrac{(L_gD_x[\gnorm{y}{}{}-B]_++ H_g\gnorm{y}{}{})^2}{L_k-2(L_f + BL_g)}\big] + \tfrac{2(H_f^2+ \sigma^2)}{L_k}\big] + \big[\Ebb[\tfrac{(L_gD_x  [\gnorm{y}{}{} - B]_++H_g\gnorm{y}{}{})^2}{L_{K-1}-2(L_f+ BL_g )}] + \tfrac{2(H_f^2+ \sigma^2)}{L_{K-1}}\big]\\
		& = \Big(\tfrac{\Ebb[(L_gD_x  [\gnorm{y}{}{} - B]_++H_g\gnorm{y}{}{})^2]}{(K-1)K(L-2(L_f+ BL_g ))} + \tfrac{2(H_f^2+ \sigma^2)}{(K-1)KL}\Big) \sum_{k=2}^{K-2} k(k+1)  + \big[\tfrac{\Ebb[(L_gD_x  [\gnorm{y}{}{} - B]_++H_g\gnorm{y}{}{})^2]}{L-2(L_f+ BL_g )} + \tfrac{2(H_f^2+ \sigma^2)}{L}\big]\\
		&  \leq \Big(\tfrac{\Ebb[(L_gD_x  [\gnorm{y}{}{} - B]_++H_g\gnorm{y}{}{})^2]}{(K-1)K(L-2(L_f+ BL_g ))} + \tfrac{2(H_f^2+ \sigma^2)}{(K-1)KL}\Big) \tfrac{ (K-1)^3}{3} + \big[\tfrac{\Ebb[(L_gD_x  [\gnorm{y}{}{} - B]_++H_g\gnorm{y}{}{})^2]}{L-2(L_f+ BL_g )} + \tfrac{2(H_f^2+ \sigma^2)}{L}\big] \\
		& \leq  \tfrac{K\Ebb[(L_gD_x  [\gnorm{y}{}{} - B]_++H_g\gnorm{y}{}{})^2]}{3(L-2(L_f+ BL_g ))} +\tfrac{\Ebb[(L_gD_x  [\gnorm{y}{}{} - B]_++H_g\gnorm{y}{}{})^2]}{L-2(L_f+ BL_g )} +  \tfrac{2K(H_f^2+ \sigma^2)}{3L}+ \tfrac{2(H_f^2+ \sigma^2)}{L}.
	\end{aligned}
\end{equation} 

Thus from \eqref{eq:first summation}, \eqref{eq:second summation}, and \eqref{eq:third summation}, the recursion in \eqref{eq:recur-nonsmooth-eta} can be written as follows 
\begin{equation}\label{eq:recur-nonsmooth-eta-4}
	\begin{aligned}
		&\Ebb(\Lcal(x_{K},s_{K}, y)- \Lcal(x,s,\bar{y}_{K})) \leq  \tfrac{26L}{(K-1)K}D_X^2 + \tfrac{1}{(K-1)\rho_{1}}\Ebb[\gnorm{y_1 - y}{}{2}]  + \tfrac{K(2(H_f^2+ \sigma^2) +\Ebb[(L_gD_x  [\gnorm{y}{}{} - B]_++H_g\gnorm{y}{}{})^2] )}{3(L-2(L_f+ BL_g ))} \\
		& \quad+ \tfrac{2(H_f^2+ \sigma^2) +\Ebb[(L_gD_x  [\gnorm{y}{}{} - B]_++H_g\gnorm{y}{}{})^2]}{L-2(L_f+ BL_g )} +\tfrac{\Ebb[(L_gD_x  [\gnorm{y}{}{} - B]_++H_g\gnorm{y}{}{})^2]}{K^2(K-1)\rho_1(M_g + M_\chi)^2}+ \tfrac{2(H_f^2 + \sigma^2)}{K^2(K-1)\rho_1(M_g + M_\chi)^2}\\
		&-\Ebb\big[\sum_{k=2}^{K-1}[\prod_{i=k+1}^{K-1}(1-\tau_{i})]\tau_k\inner{\delta_k}{\dot{x}_k-x}\big]  - \tfrac{2\Ebb[\inner{\delta_1}{\dot{x}_1-x}]}{K(K-1)}  .
	\end{aligned}
\end{equation}
We know that for any $(x_K,s_K,y)\in X\times\mathbb{R}^m_-\times\mathbb{R}^m_+$, $ \Lcal(x_K,y)\leq \Lcal(x_K,s_K,y)$ and for $s =0$ we have $\Lcal(x,\bar{y}_K) = \Lcal(x,0,\bar{y}_K)$. Therefore, we have the following
\begin{equation}\label{eq:conjugate-1}
	\Lcal(x_K,y) - \Lcal(x,\bar{y}_K) \leq \Lcal(x_K,s_K,y) - \Lcal(x,0,\bar{y}_K) \quad \forall (x,y)\in X\times \mathbb{R}^m_+.
\end{equation}
Now, let us choose $ y =0 $. we know $ \bar{y}_K\geq 0$ and $\psi(x^*)\leq 0$ thus $ \langle \bar{y}_K, \psi(x^*)\rangle \leq 0 $. Also note that $\Ebb[\inner{\delta_k}{\dot{x}_k-x^*}] = 0$. Therefore, regarding \eqref{eq:recur-nonsmooth-eta-4} and \eqref{eq:conjugate-1}, we have
\begin{equation*}\label{eq:optimality-gap}
	\Ebb[\psi_0(x_K)- \psi_0(x^*)]\leq\tfrac{26L}{(K-1)K}D_X^2+ \tfrac{1}{(K-1)\rho_{1}}\gnorm{y_1}{}{2} + \tfrac{2(K+3)(H_f^2+ \sigma^2)}{3(L-2(L_f+ BL_g ))}+ \tfrac{2(H_f^2 + \sigma^2)}{K^2(K-1)\rho_1(M_g + M_\chi)^2},
\end{equation*}
Applying the step size policy in \eqref{eq:non-ergo-stepsize-non-smooth-eta} we obtain 
\begin{equation*}
	\begin{aligned}
		\Ebb[\psi_0(x_K)- \psi_0(x^*)]\leq&\tfrac{13D_X\sqrt{120K(\mathcal{H}_*^2 + 2(H_f^2+ \sigma^2))}}{60(K-1)}+ \tfrac{1}{(K-1)\rho_{1}}\gnorm{y_1}{}{2} + \tfrac{52D_X^2(L_f+ BL_g +\rho_{1} K(M_g + M_\chi)^2)}{(K-1)K}\\
		& \quad+\tfrac{80D_X(K+3)(H_f^2+ \sigma^2)}{K(240D_X\rho_{1}(M_g+M_\chi)^2 + \sqrt{120K(\mathcal{H}_*^2 + 2(H_f^2+\sigma^2))})}+ \tfrac{2(H_f^2 + \sigma^2)}{K^2(K-1)\rho_1(M_g + M_\chi)^2}.
	\end{aligned}
\end{equation*}
Now, let us move to the feasibility gap. First, observe that 
\begin{equation*}\label{eq:first-frasib}
	\Lcal(x_K,y^*) - \Lcal(x^*,y^*)\geq 0,
\end{equation*}

implying that 
\begin{equation*}\label{eq:second-feasib}
	\psi_0(x_K) + \inner{y^*}{\psi(x_K)} - \psi_0(x^*)\geq 0 ,
\end{equation*}

which in view of the below relation 
\begin{equation*}\label{eq:feasib3}
	\inner{y^*}{\psi(x_K)}\leq \inner{y^*}{[\psi(x_K)]_+}\leq \gnorm{y^*}{}{}\gnorm{[\psi(x_K)]_+}{}{},
\end{equation*}
we have 
\begin{equation}\label{eq:haty}
	\psi_0(x_K) +\gnorm{y^*}{}{}\gnorm{[\psi(x_K)]_+}{}{} - \psi_0(x^*)\geq 0.
\end{equation}
Moreover, $\Lcal(x_K,\hat{y})-\Lcal(x^*,\bar{y}_K)\geq\Lcal(x_K,\hat{y})-\Lcal(x^*,y^*) = \psi_0(x_K) +(\gnorm{y^*}{}{}+1)\gnorm{[\psi(x_K)]_+}{}{} - \psi_0(x^*)$ where we used the definition of $\hat{y}$ as $\hat{y} = (\gnorm{y^*}{}{}+1)\tfrac{[\psi(\bar{x}_K)]_+}{\gnorm{[\psi(x_K)]_+}{}{}}$. Therefore,  considering \eqref{eq:conjugate-1}, \eqref{eq:haty} and $\Ebb[\inner{\delta_k}{\dot{x}_k-x^*}] = 0$, the feasibility gap can be written as 
\begin{equation*}
	\begin{aligned}
		\Ebb(\gnorm{[\psi(x_K)]_+}{}{})\leq  \Ebb(\Lcal(x_K,\hat{y})- &\Lcal(x^*,\bar{y}_K)) \leq  \tfrac{13D_X\sqrt{120K(\mathcal{H}_*^2 + 2(H_f^2+ \sigma^2))}}{6(K-1)} + \tfrac{52D_X^2(L_f+ BL_g + \rho_{1}K(M_g + M_\chi)^2)}{(K-1)K}\\
		& \quad + \tfrac{40D_X(K+3)(\mathcal{H}_*^2+2(H_f^2+ \sigma^2) )}{K(240D_X\rho_{1}(M_g+M_\chi)^2 + \sqrt{120K(\mathcal{H}_*^2 + 2(H_f^2+\sigma^2))})}+ \tfrac{1}{(K-1)\rho_{1}}\gnorm{y_1 - \hat{y}}{}{2} \\
		&\quad +\tfrac{\mathcal{H}_*^2+2(H_f^2 + \sigma^2)}{K^2(K-1)\rho_1(M_g + M_\chi)^2}.
	\end{aligned}
\end{equation*}  
\subsection{Convergence analysis for the strongly convex case}
Before stating the convergence results for the strongly convex case, let us mention the following lemma  
\begin{lemma}\label{lemma:Dinh's strong}
	(Lemma 15 of \cite{zhu2022new}) Given \( \mu_f \geq 0 \), \( L_k \geq \mu_f \), and \( \tau_{k} \in (0,1) \), let \( m_k := \tfrac{L_k}{L_{k-1}} \) and \( \{x_k\}\) be a given sequence in \( \mathbb{R}^n \). We define \( \hat{x}_k := x_k + \beta_k \left( x_k - x_{k-1} \right) \), where \( \beta_k := \tfrac{(1 - \tau_{k-1})\tau_{k-1}}{\tau_{k-1}^2+m_k\tau_{k}} \).
	
	Assume that the following two conditions hold
	\begin{subequations}\label{eq:lemma15-cond}
		\begin{equation}
			L_{k-1} (1-\tau_{k})\tau_{k-1}^2 + \mu_f(1-\tau_{k})\tau_{k} \geq (L_k-\mu_f)\tau_{k}^2,
		\end{equation}
		\begin{equation}
			L_{k-1} (\tau_{k-1}^2+ m_k\tau_{k})m_k\tau_{k}\geq (L_k-\mu_f)\tau_{k-1}^2.
		\end{equation}
	\end{subequations}
	Then, for any \( x \in \mathbb{R}^n \), we have
	\[
	(L_k-\mu_f)\tau_{k}^2\gnorm{\tfrac{1}{\tau_{k}}[\hat{x}_k-(1-\tau_{k})x_k]-x}{}{2} - \mu_f\tau_{k}(1-\tau_{k})\gnorm{x_k-x}{}{2}\leq 	(1-\tau_{k})L_{k-1} \tau_{k-1}^2\gnorm{\tfrac{1}{\tau_{k-1}}[x_k-(1-\tau_{k-1})x_{k-1}]-x}{}{2}. 
	\]
	
\end{lemma}
Now, we are ready to mention the proof of Theorem \ref{thm:non-ergo-strong-semi}. 	
\paragraph{Proof of Theorem \ref{thm:non-ergo-strong-semi}.}
First of all, the updates in \eqref{eq:non-ergo-stepsize-strong} leads to the following equations for all $k\geq 2$
\begin{equation*}\label{eq:relations-last-iter}
	\tau_k^2 = (1-\tau_k)\tau_{k-1}^2, \quad \tfrac{1}{k+1}\leq \tau_{k}\leq\tfrac{2}{k+1}.
\end{equation*}
Thus, $\rho_{k-1} = \tfrac{\rho_{1}}{\tau_{k-1}^2} = \tfrac{(1-\tau_k)\rho_1}{\tau_k^2} = (1-\tau_k)\rho_k$ which implies 
\begin{equation*}\label{eq:relations-last-iter-2}
	\tfrac{1}{\eta_k}= \tfrac{1-\tau_k}{\eta_{k-1}}, \quad \forall k\geq 2.
\end{equation*}
Also, note $\eta_k<2\rho_k $ and $L_{k}$ satisfies the conditions in Lemma \ref{lem:F-contraction} and Lemma \ref{lemma:importantlemma} implying all the conditions in these lemmas hold. In addition, since $0< \rho_1=\tfrac{\mu_f}{2(M_g + M_\chi)^2}$, and $\{\beta_{k}\}_{k\geq2}$ in \eqref{eq:non-ergo-stepsize-strong} is defined as Lemma \ref{lemma:Dinh's strong}, then all conditions in Lemma \ref{lemma:Dinh's strong} including \eqref{eq:lemma15-cond} hold for all $k\geq 2$.
Hence using Lemma \ref{lemma:Dinh's strong}, we have the following useful relation for any $x\in X$. 
\begin{equation}\label{eq:resultlemma15}
	\begin{aligned}
		&(L_k-\mu_f)\tau_k^2\gnorm{\tfrac{1}{\tau_k}[\wh{x}_k-(1-\tau_k)x_k] -x}{}{2}-\mu_f(1-\tau_k)\tau_k\gnorm{x_k-x}{}{2}\\
		\leq &(1-\tau_k)L_{k-1}\tau_{k-1}^2\gnorm{\tfrac{1}{\tau_{k-1}}[x_{k} - (1-\tau_{k-1})x_{k-1}] - x}{}{2}, \quad \forall k\geq 2,
	\end{aligned}
\end{equation}
Now, substituting \eqref{eq:resultlemma15} into Lemma \ref{lemma:importantlemma}, one can obtain the following result for all $k\geq 2$
\begin{equation}\label{eq:recur1}
	\begin{aligned}
		\Lcal(x_{k+1}, s_{k+1}, y) -& \Lcal(x,s,\wb{y}_{k+1})  
		\le(1-\tau_k)[\Lcal(x_k, s_k, y) - \Lcal(x,s,\wb{y}_k)]\\
		&+\tfrac{(1-\tau_k)L_{k-1}\tau_{k-1}^2}{2}\gnorm{\tfrac{1}{\tau_{k-1}}[x_{k} - (1-\tau_{k-1})x_{k-1}] - x}{}{2}\\
		&\quad - \tfrac{L_k\tau_k^2}{2}\gnorm{\tfrac{1}{\tau_k}[x_{k+1} - (1-\tau_k)x_k] - x}{}{2}   + \tfrac{1-\tau_k}{2\eta_{k-1}}\gnorm{y-\wt{y}_k}{}{2} -\tfrac{1}{2\eta_k} \gnorm{y-\wt{y}_{k+1}}{}{2}\\
		&\quad  + \tfrac{(L_gD_x[\gnorm{y}{}{}-B]_+ + H_g\gnorm{y}{}{})^2}{L_k-2(L_f + BL_g)}+ \tfrac{2(H_f^2+\gnorm{\delta_k}{}{2})}{L_k}-\tau_k\inner{\delta_k }{\dot{x}_k-x}.
	\end{aligned}
\end{equation}
By rearranging \eqref{eq:recur1} we have the following result 
\begin{equation}\label{eq:recur2}
	\begin{aligned}
		\Lcal(x_{k+1}, &s_{k+1}, y) - \Lcal(x,s,\wb{y}_{k+1})+\tfrac{L_k\tau_k^2}{2}\gnorm{\tfrac{1}{\tau_k}[x_{k+1} - (1-\tau_k)x_k] - x}{}{2}+ \tfrac{1}{2\eta_k} \gnorm{y-\wt{y}_{k+1}}{}{2}  \\
		&\le(1-\tau_k)[\Lcal(x_k, s_k, y) - \Lcal(x,s,\wb{y}_k)+\tfrac{L_{k-1}\tau_{k-1}^2}{2}\gnorm{\tfrac{1}{\tau_{k-1}}[x_{k} - (1-\tau_{k-1})x_{k-1}] - x}{}{2}\\
		&\quad + \tfrac{1}{2\eta_{k-1}}\gnorm{y-\wt{y}_k}{}{2} ] -\tau_{k}\inner{\delta_k }{\dot{x}_k-x} + \tfrac{(L_gD_x[\gnorm{y}{}{}-B]_+ + H_g\gnorm{y}{}{})^2}{L_k-2(L_f + BL_g)}+ \tfrac{2(H_f^2+\gnorm{\delta_k}{}{2})}{L_k} .
	\end{aligned}
\end{equation}
Let us consider $Q_k(x,y)$ as follows 
\[
Q_k(x,y)  = \Lcal(x_k, s_k, y) - \Lcal(x,s,\wb{y}_k)+\tfrac{L_{k-1}\tau_{k-1}^2}{2}\gnorm{\tfrac{1}{\tau_{k-1}}[x_{k} - (1-\tau_{k-1})x_{k-1}] - x}{}{2}+ \tfrac{1}{2\eta_{k-1}}\gnorm{y-\wt{y}_k}{}{2}.
\]
Hence \eqref{eq:recur2} implies the following 
\begin{equation}\label{eq:Q_recur}
	Q_{k+1}(x,y)\leq (1-\tau_k)Q_k(x,y)  +  \tau_{k}\inner{\delta_k }{x-\dot{x}_k} + \tfrac{(L_gD_x[\gnorm{y}{}{}-B]_++H_g\gnorm{y}{}{})^2}{L_k-2(L_f + BL_g)}+ \tfrac{2(H_f^2+\gnorm{\delta_k}{}{2})}{L_k}, \quad \forall k\geq 2.
\end{equation}
Moreover regarding  initialization in Algorithm \ref{alg:augConEx}, Lemma \ref{lemma:importantlemma} and the step-size policy in Theorem \ref{thm:non-ergo-strong-semi}, we have the following relations for $k=2$ 
\begin{equation}\label{eq:Q_1}
	Q_2(x,y)\leq \tfrac{L_1-\mu_f}{2}\gnorm{x_1-x}{}{2} + \tfrac{1}{2\eta_{1}}\gnorm{y_1-y}{}{2}+ \tfrac{(L_gD_x[\gnorm{y}{}{}-B]_+ + H_g\gnorm{y}{}{})^2}{L_1-2(L_f + BL_g)}+ \tfrac{2(H_f^2+\gnorm{\delta_1}{}{2})}{L_1}+\inner{\delta_1 }{x-\dot{x}_1}.
\end{equation}
Using the recursion in \eqref{eq:Q_recur} from $k=2$ to $K-1$ and considering $\tau_1 =1, \tilde{y}_1 = y_1$, and $ \hat{x}_1 = x_1$ we have 
\begin{equation}\label{eq:recur3}
	\begin{aligned}
		\Lcal(x_{K}, &s_{K}, y) - \Lcal(x,s,\wb{y}_{K}) 
		\le[\prod_{k=2}^{K-1} (1-\tau_{k})]Q_2(x,y)
		\\
		&\quad +\sum_{k=2}^{K-1}\big[ [\prod_{i=k+1}^{K-1} (1-\tau_{i})]\tau_{k} \inner{\delta_k }{x-\dot{x}_k} + \tfrac{(L_gD_x[\gnorm{y}{}{}-B]_++H_g\gnorm{y}{}{})^2}{L_k-2(L_f + BL_g)}+ \tfrac{2(H_f^2+\gnorm{\delta_1}{}{2})}{L_k}\big]\\
		&\leq \tfrac{4}{K^2}[(L_f+BL_g)\gnorm{x_1-x}{}{2}+ \tfrac{1}{2\eta_{1}}\gnorm{y_1-y}{}{2} + \tfrac{(L_gD_x[\gnorm{y}{}{}-B]_++H_g\gnorm{y}{}{})^2}{2\rho_{1}(M_g+M_\chi)^2}\\
		&\quad+\tfrac{H_f^2+\gnorm{\delta_1}{}{2}}{(L_f+BL_g+\rho_{1}(M_g + M_\chi^2))} +\inner{\delta_1}{x-\dot{x}_1}] +\inner{\delta_{K-1}}{x-\hat{x}_{K-1}} + \tfrac{(L_gD_x  [\gnorm{y}{}{} - B]_++H_g\gnorm{y}{}{})^2}{L_{K-1}-2(L_f+ BL_g )} + \tfrac{2(H_f^2+ \gnorm{\delta_{K-1}}{}{2})}{L_{K-1}}\\
		&\quad	 +\sum_{k=2}^{K-2}\big[ [\prod_{i=k+1}^{K-1} (1-\tau_{i})] \tau_{k}\inner{\delta_k }{x-\hat{x}_k} + \tfrac{(L_gD_x[\gnorm{y}{}{}-B]_++H_g\gnorm{y}{}{})^2}{L_k-2(L_f + BL_g)}+ \tfrac{2(H_f^2+\gnorm{\delta_k}{}{2})}{L_k}\big],
	\end{aligned}
\end{equation}
where we used \eqref{eq:Q_1}, $2\rho_{1}(M_g+M_{\chi})^2= \mu_f$  and $[\prod_{k=2}^{K-1} (1-\tau_{k})] = \tfrac{\tau_{K-1}^2}{\tau_1}\leq \tfrac{4}{K^2}$ to obtain the second inequality. 	
We know that for any $y\geq \textbf{0}$, $ \Lcal(x_K,y)\leq \Lcal(x_K,s_K,y)$ and for $s =0 $, we have $\Lcal(x,\bar{y}_K) = \Lcal(x,0,\bar{y}_K)$ leading to the following relation 
\begin{equation}\label{eq:relation3}
	\Lcal(x_K,y) - \Lcal(x,\bar{y}_K) \leq \Lcal(x_K,s_K,y) - \Lcal(x,0,\bar{y}_K).
\end{equation}
Using \eqref{eq:relation3}, with \eqref{eq:recur3} alongside the definition of $  \Lcal(x,y)$ we have 
\begin{equation}\label{eq:non-ergo-8}
	\begin{aligned}
		&\psi_0(x_K)+ \langle y , \psi(x_K)\rangle - \psi_0(x) - \langle \bar{y}_K, \psi(x)\rangle \leq  \tfrac{4}{K^2}[(L_f+BL_g)\gnorm{x_1-x}{}{2}\\
		&\quad+ \tfrac{1}{2\eta_{1}}\gnorm{y_1-y}{}{2}+\tfrac{(L_gD_x[\gnorm{y}{}{}-B]_++H_g\gnorm{y}{}{})^2 +2(H_f^2+\gnorm{\delta_1}{}{2})}{2\rho_{1}(M_g+M_\chi)^2}+\inner{\delta_1}{x-\dot{x}_1}]
		\\&\quad +\sum_{k=2}^{K-1}\big[ [\prod_{i=k+1}^{K-1} (1-\tau_{i})] \tau_{k}\inner{\delta_k }{x-\dot{x}_k} + \tfrac{(L_gD_x[\gnorm{y}{}{}-B]_++H_g\gnorm{y}{}{})^2}{L_k-2(L_f + BL_g)}+ \tfrac{2(H_f^2+\gnorm{\delta_k}{}{2})}{L_k}\big].
	\end{aligned}
\end{equation}
Now let us choose $ y =0 $. we know $ \bar{y}_K\geq 0$ and $\psi(x^*)\leq 0$, thus $ \langle \bar{y}_K, \psi(x^\ast)\rangle \leq 0 $. Rewriting \eqref{eq:non-ergo-8} for $(x,y)=(x^*,0)$ and taking the expectation from both sides with respect to $\Ebb[\inner{\delta_k}{\dot{x}_k-x^*}] = 0 $, one can conclude the following 
\begin{equation}\label{eq:feasib}
	\begin{aligned}
		\Ebb[\psi_0(x_K) - \psi(x^*)]\leq& \tfrac{4}{K^2}[4(L_f+BL_g)D_X^2
		+ \tfrac{1}{2\eta_{1}}\gnorm{y_1}{}{2} + \tfrac{(H_f^2+\sigma^2)}{\rho_{1}(M_g+M_\chi)^2} ]\\
		&\quad	+\sum_{k=2}^{K-2}\big[ [\prod_{i=k+1}^{K-1} (1-\tau_{i})]  \tfrac{2(H_f^2+\sigma^2)}{L_k}\big] + \tfrac{H_f^2+ \sigma^2}{\rho_{K-1}(M_g+M_\chi)^2}.
	\end{aligned}
\end{equation}
Now let us focus on the following summation 
\begin{equation}\label{eq:first-sum}
	\begin{aligned}
		\sum_{k=2}^{K-2} [\prod_{i=k+1}^{K-1} (1-\tau_{i})]\tfrac{2(H_f^2+\sigma^2)}{L_k}
		&\leq \sum_{k=2}^{K-2} \tfrac{\tau_{K-1}^2}{\tau_{k}^2}\tfrac{H_f^2+\sigma^2}{\rho_k (M_g+M_\chi)^2} \leq \tfrac{4(H_f^2+\sigma^2)}{K^2}\sum_{k=2}^{K-2} \tfrac{1}{\rho_1 (M_g+M_\chi)^2} \leq\tfrac{8(H_f^2+\sigma^2)}{\mu_fK},
	\end{aligned}
\end{equation}
where we used $1-\tau_k = \tfrac{\tau_k^2}{\tau_{k-1}^2}$, $\tfrac{1}{k+1}\leq \tau_k\leq\tfrac{2}{k+1}$, $\tau_k^2\rho_k = \rho_1$ and $2\rho_{1}(M_g+M_\chi)^2=\mu_f$ in the last two inequalities. Moreover, note that $\tfrac{H_f^2+ \sigma^2}{\rho_{K-1}(M_g+M_\chi)^2} \leq \tfrac{8(H_f^2+\sigma^2)}{\mu_fK^2}$.Therefore, using \eqref{eq:feasib} and \eqref{eq:first-sum}, we concludes \eqref{eq:opt-1-non-ergo}.
Now let us focus on the feasibility convergence rate. First, we rewrite \eqref{eq:non-ergo-8} for any $ x=x^*$ and $\hat{y}$, where $\hat{y} = (\gnorm{y}{}{}+1)\tfrac{[\psi(\bar{x}_K)]_+}{\gnorm{[\psi(\bar{x}_K)]_+}{}{}}$. We have 
\begin{equation}\label{eq:non-ergo-10}
	\begin{aligned}
		\psi&_0(x_K)+ \langle \hat{y} , \psi(x_K)\rangle - \psi_0(x^*) - \langle \bar{y}_K, \psi(x^*)\rangle \leq  \tfrac{4}{K^2}[(L_f+BL_g) \gnorm{x_1-x^*}{}{2}\\
		&+ \tfrac{1}{2\eta_{1}}\gnorm{y_1-\hat{y}}{}{2}+\tfrac{\mathcal{H}_*^2 +2(H_f^2+\gnorm{\delta_1}{}{2})}{2\rho_{1}(M_g+M_\chi)^2} - \inner{\delta_1}{\dot{x}_1-x^*}]+ \big[\inner{\delta_{K-1}}{x^*-\dot{x}_{K-1}}+\tfrac{\mathcal{H}_*^2}{L_{K-1}-2(L_f+ BL_g )}  + \tfrac{2(H_f^2+ \gnorm{\delta_k}{}{2})}{L_{K-1}}\big]\\
		&\quad
		+\sum_{k=1}^{K-2}\big[ [\prod_{i=k+1}^{K-1} (1-\tau_{i})] \inner{\delta_k }{x^*-\hat{x}_k} + \tfrac{\mathcal{H}_*^2}{L_k-2(L_f + BL_g)}+ \tfrac{2(H_f^2+\gnorm{\delta_k}{}{2})}{L_k}\big].\\
	\end{aligned}
\end{equation}
Using the similar argument in Theorem \ref{thm:convex-case}, one can conclude the following
\begin{equation}\label{eq:non-ergo-feas-10}
	\begin{aligned}
		\gnorm{[ \psi({x}_K)]_+}{}{}& \leq  \tfrac{4}{K^2}[(L_f+BL_g)\gnorm{x_1-x^*}{}{2}+ \tfrac{1}{2\eta_{1}}\gnorm{y_1-\hat{y}}{}{2} 
		+ \tfrac{\mathcal{H}_*^2 + 2(H_f^2+\|\delta_k\|^2)}{2\rho_{1}(M_g+M_\chi)^2}+ \inner{\delta_1}{x-\dot{x}_1}]\\
		&\quad+\inner{\delta_{K-1}}{x-\hat{x}_{K-1}}+\big[\tfrac{\mathcal{H}_*^2}{L_{K-1}-2(L_f+ BL_g )} + \tfrac{2(H_f^2+ \gnorm{\delta_{K-1}}{}{2})}{L_{K-1}}\big]\\
		&\quad+\sum_{k=1}^{K-2}\big[ [\prod_{i=k+1}^{K-1} (1-\tau_{i})] \inner{\delta_k }{x-\hat{x}_k} + \tfrac{\mathcal{H}_*^2}{L_k-2(L_f + BL_g)}+ \tfrac{2(H_f^2+\gnorm{\delta_k}{}{2})}{L_k}\big].
	\end{aligned}
\end{equation} 
Taking the expectation from both sides of \eqref{eq:non-ergo-feas-10} with respect to $\Ebb[\inner{\delta_k}{\dot{x}_k-x^*}] = 0 $, using similar relations we had in \eqref{eq:first-sum}, and $2\rho_{1}(M_g+M_\chi)^2= \mu_f$ one concludes \eqref{eq:feasib-1-non-ergo}.

Lastly, let us move to the last iterate convergence. One can use \eqref{eq:Q_recur} construct the following inequality for $(x,s,y) = (x^*,s^*,y^*)$. Note that we know $\Lcal(x_{K}, s_{K}, y^*) - \Lcal(x^*,s^*,\wb{y}_{K})\geq 0$. Therefore we can ignore the part of \eqref{eq:Q_recur} related to the gap function and say 
\begin{equation}\label{eq:recur4}
	\begin{aligned}
	\tfrac{L_{K-1}\tau_{K-1}^2}{2}&\gnorm{\tfrac{1}{\tau_{K-1}} [x_{K} - (1-\tau_{K-1})x_{K-1}] - x^*}{}{2}
		\le[\prod_{k=2}^{K-1} (1-\tau_{k})]Q_2(x^*,y^*)
		\\
		&\quad +\sum_{k=2}^{K-1}\big[ [\prod_{i=k+1}^{K-1} (1-\tau_{i})]\tau_{k} \inner{\delta_k }{x^*-\dot{x}_k} + \tfrac{(L_gD_x[\gnorm{y^*}{}{}-B]_++H_g\gnorm{y^*}{}{})^2}{L_k-2(L_f + BL_g)}+ \tfrac{2(H_f^2+\gnorm{\delta_k}{}{2})}{L_k}\big]\\
		&\leq \tfrac{4}{K^2}[(L_f+BL_g)\gnorm{x_1-x^*}{}{2}+ \tfrac{1}{2\eta_{1}}\gnorm{y_1-y^*}{}{2}+ \tfrac{(L_gD_x[\gnorm{y^*}{}{}-B]_++H_g\gnorm{y^*}{}{})^2}{2\rho_{1}(M_g+M_\chi)^2}\\
		&\quad+\tfrac{H_f^2+\gnorm{\delta_1}{}{2}}{(L_f+BL_g+\rho_{1}(M_g + M_\chi^2))} +\inner{\delta_1}{x^*-\dot{x}_1}] +\inner{\delta_{K-1}}{x^*-\hat{x}_{K-1}} + \tfrac{(L_gD_x  [\gnorm{y^*}{}{} - B]_++H_g\gnorm{y^*}{}{})^2}{L_{K-1}-2(L_f+ BL_g )}\\
		&\quad	 + \tfrac{2(H_f^2+ \|\delta_{K-1}\|^2)}{L_{K-1}} +\sum_{k=2}^{K-2}\big[ [\prod_{i=k+1}^{K-1} (1-\tau_{i})] \tau_{k}\inner{\delta_k }{x^*-\hat{x}_k} + \tfrac{(L_gD_x[\gnorm{y^*}{}{}-B]_++H_g\gnorm{y^*}{}{})^2}{L_k-2(L_f + BL_g)}+ \tfrac{2(H_f^2+\gnorm{\delta_k}{}{2})}{L_k}\big].
	\end{aligned}
\end{equation}
We obtain \eqref{eq:last-iter} similar to the process of deriving \eqref{eq:feasib-1-non-ergo} so we avoid repetition. 
$ \openbox $



\section{Numerical Experiments}\label{sec:numerical}
This section is dedicated to numerically validating the performance of the $ \augconex $ method. The experiments are coded in MATLAB R2024b, running on a  64-bit PC with 2.1 GHz Intel Core i7-12700 and 32Gb RAM. We implement the \augconex~method on a sparse quadratically constrained quadratic program (QCQP). In particular, we have
\begin{equation}\label{eq:sparse}
    \min_{\gnorm{x}{}{}\leq D_x}\{\psi_0: = \tfrac{1}{2}x^\top A_0 x + b_0^\top x + \lambda\gnorm{x}{1}{}\text{ s.t. } \tfrac{1}{2}x^\top A_i x + b_i^\top x - c_i\leq 0, \text{ } i\in [m]\},
\end{equation}
where $A_i\in \mathbb{R}^{n\times n}$, $i = 0,1,\dots,m$ are randomly generated positive semi-definite matrices, and $b_i \in \mathbb{R}^{n}, i = 0,1,\dots, m$ are random vectors. Moreover, $c_i, i\in [m]$ are random numbers drawn from a uniform distribution on the interval $[0,2]$. This choice of $c$ makes $\zero$ a strictly feasible solution, which also ensures the existence of a dual solution. Note that, in this setting, the sparse QCQP problem \eqref{eq:sparse} is a subclass of the composite nonsmooth convex function-constrained problem \eqref{eq:prima-problem}. We compare the performance of the last iterate of \augconex~method with the average iterate of the ConEx method \cite{boob2023stochastic} for solving \eqref{eq:sparse}. The ConEx method also applies to composite nonsmooth and stochastic problems and provides convergence guarantee on the average-iterate \cite[Theorem 1 and 2]{boob2023stochastic}. In this experiment, the sparsity level is determined by the penalty parameter $\lambda$. It is well-known that as we increase $\lambda$, we expect that the optimal solution of \eqref{eq:sparse} will get more sparse.

For these experiments, we set $n=100$ and $m=10$. Furthermore, $D_x =10$, $L_f = \gnorm{A_0}{}{}$, $L_{g_i} = \gnorm{A_i}{}{}$, $M_{g_i} = D_x\gnorm{A_i}{}{} + \gnorm{b_i}{}{}$, and $M_{\chi_i} = 0$ for $i\in [m]$. Also, note since we have quadratic constraints, $H_g = 0$. To verify the performance of both algorithms in the stochastic setting, we artificially make the gradients noisy as follows: Suppose $v \in \Rbb^n$ is the true gradient of $f$ and $\xi \sim N(\zero, I_n)$ be the random vector of $n$ independent and identically distributed (i.i.d) standard normal random variables. Then, we use the noisy gradient $\tilde{v} = v + \sigma \xi$ where $\sigma$ is the standard deviation of the noise. It is clear that $\tilde{v}$ is an unbiased estimator of $v$, satisfying the requirements of our setting (see  \eqref{eq:SO_F_objective}).  We use $\sigma = 10$ throughout the experiments reported below. 

We use both Aug-ConEx and ConEx methods to solve \eqref{eq:sparse} in two possible ways. First, we treat $\ell_1$-norm regularization $\lambda \gnorm{x}{1}{}$ as a separate function $\chi_0$, leading to having a proximal operator in Algorithm \ref{alg:implicit}. In the second variant,  we consider $\ell_1$-norm regularization as a part of $f$ resulting in the linearization of $\lambda \gnorm{x}{1}{}$ as a part of $f$ first-order approximation in \eqref{eq:conceptual-x-s-update}. In this variant, we will have a simple projection onto an $\ell_2$-ball with radius $D_x$ in Algorithm \ref{alg:implicit}. Note that for the first and second variants, we have $H_f=0$, $H_f = 2\lambda\sqrt{n}$ respectively. In each of those settings, we compare the performance of $\augconex$ and ConEx in terms of feasibility and optimality gaps for convex and strongly convex cases. 

 Figure \ref{fig:prox}, in particular, Figure \ref{fig:subfig_prox1} and Figure \ref{fig:subfig_prox2}, compare $\augconex$ and ConEx methods in optimality and feasibility gap metrics against iteration and run time for 10 i.i.d instances for both convex and strongly convex problems, respectively. In particular, we report the average optimality and feasibility gaps across the 10 random instances. Moreover, for this experiment, we run both methods in the first variant where we treat $\ell_1$-norm as a prox-friendly function over an $\ell_2$-ball. We did not see any works on the evaluation of such prox operators in the literature. Hence, we show that it can be computed explicitly in Proposition \ref{prop:sparse}.
\begin{figure}[h]
    \centering
    \begin{subfigure}{.45\textwidth}
        \centering
        \includegraphics[width=\linewidth]{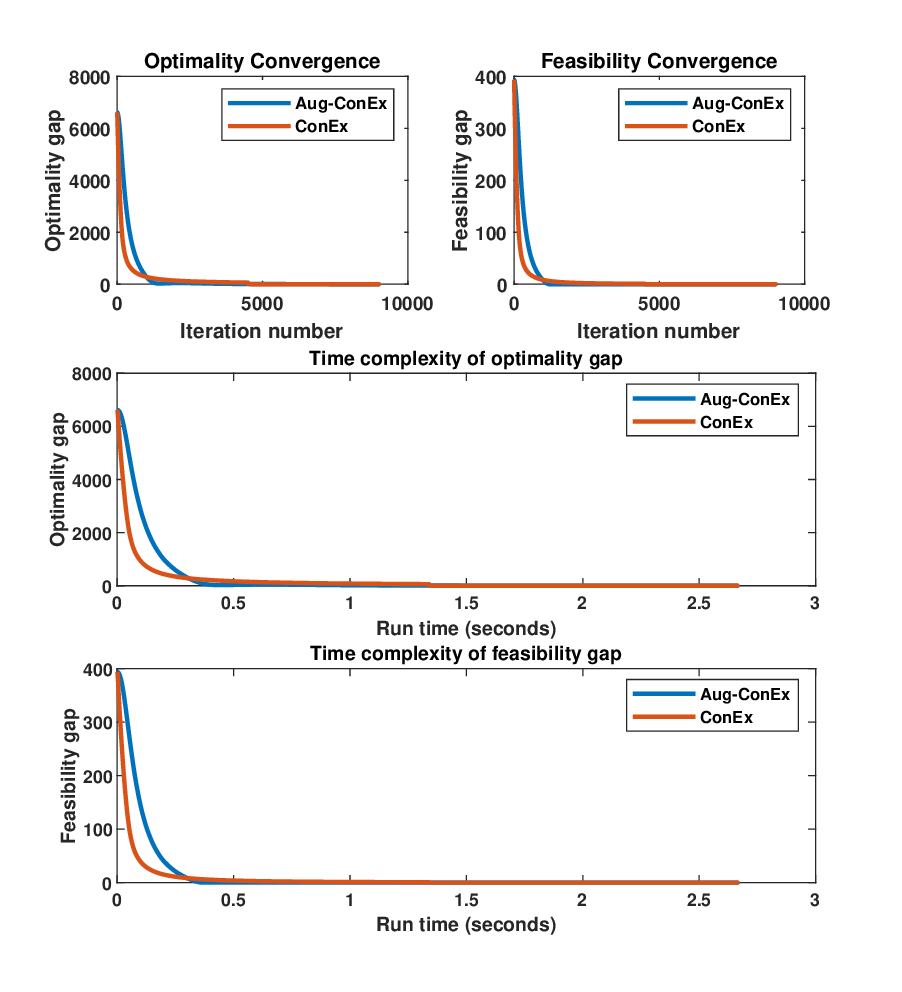}  
        \caption{Convex setting for 9000 iterations}
        \label{fig:subfig_prox1}
    \end{subfigure}
    \hfill
    \begin{subfigure}{0.45\textwidth}
        \centering
        \includegraphics[width=\linewidth]{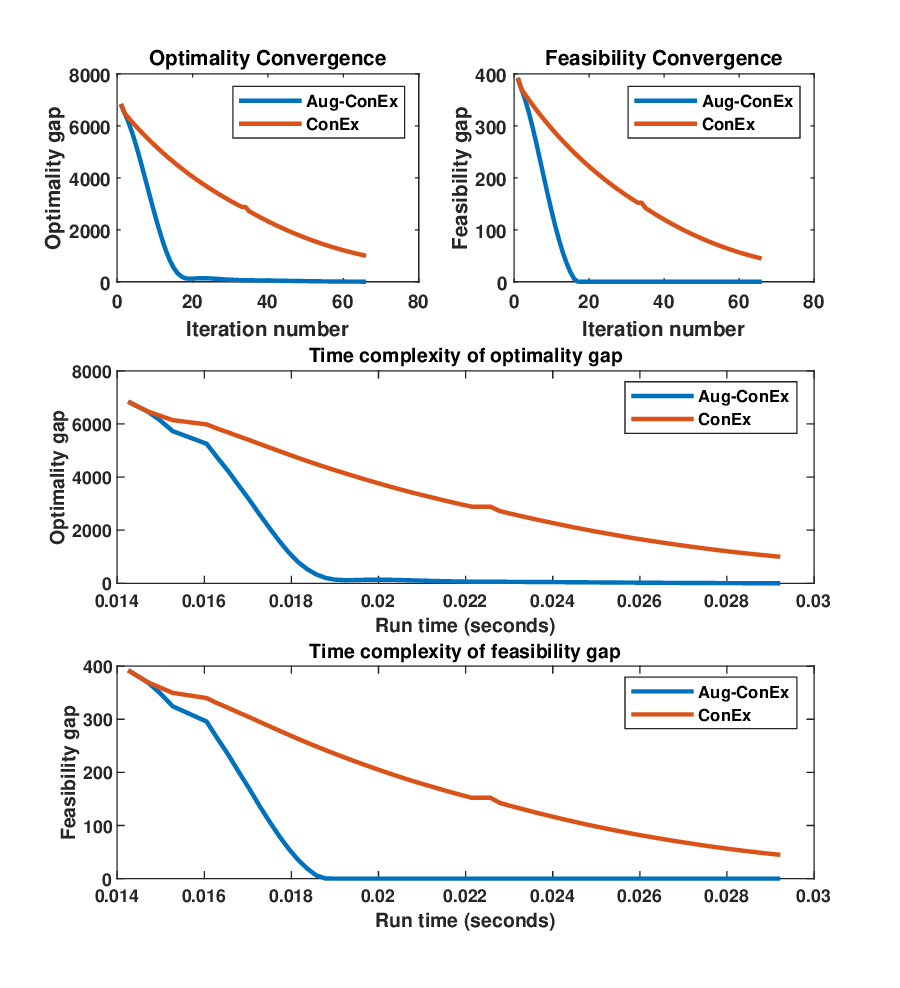}  
        \caption{Strongly convex setting for 66 iterations}
        \label{fig:subfig_prox2}
    \end{subfigure}
    
    \caption{Performance of $\augconex$ vs ConEx in terms of optimality and feasibility gaps on 10 i.i.d instances of sparse QCQP using the proximal operator in Algorithm \ref{alg:implicit}.}
    \label{fig:prox}
\end{figure}
As one can observe from Figure \ref{fig:subfig_prox1}, ConEx method converges slightly faster than the $\augconex$ method in the initial phase. However, the optimality and feasibility gap achieved by \augconex~is slightly better than the ConEx method. The situation changes substantially in the strongly convex setting. Indeed, in Figure \ref{fig:subfig_prox2}, we see that $\augconex$ has much faster convergence than the ConEx method in both optimality and feasibility metrics against iteration and run times.  Now, let us compare these two methods in the second variant of the experiments where we have a simple projection in computing the implicit step of $\augconex$ in Algorithm \ref{alg:implicit}. As we can see from Figure \ref{fig:proj}, the results are similar to the first variant in terms of convergence rate in convex and strongly convex settings. 
\begin{figure}[h]
    \centering
    \begin{subfigure}{.45\textwidth}
        \centering
        \includegraphics[width=\linewidth]{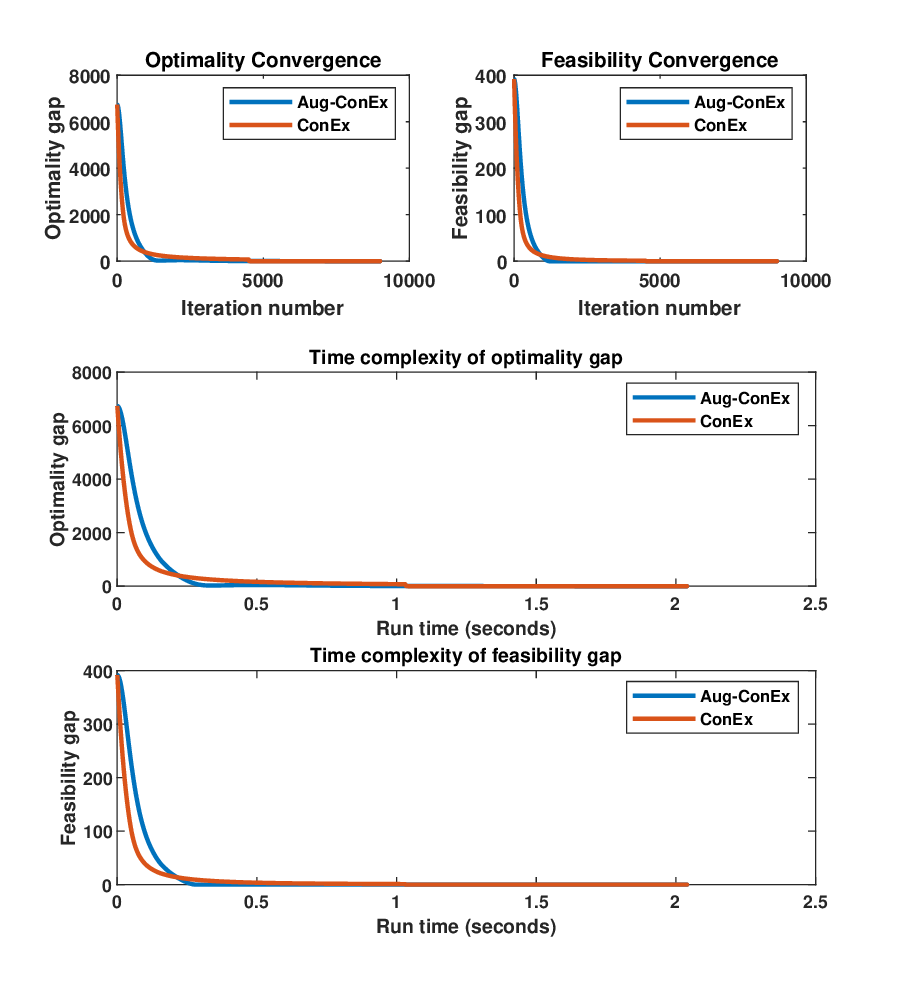}  
        \caption{Convex setting for 9000 iterations}
        \label{fig:subfig1}
    \end{subfigure}
    \hfill
    \begin{subfigure}{0.45\textwidth}
        \centering
        \includegraphics[width=\linewidth]{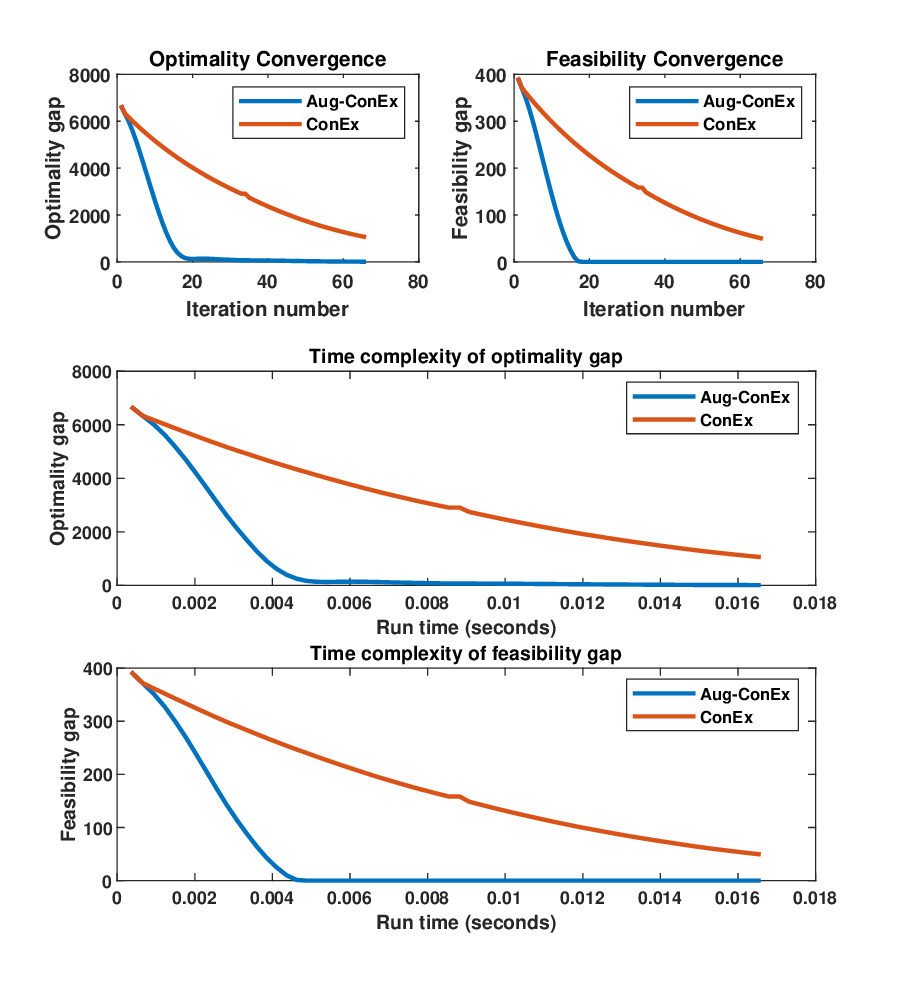}  
        \caption{Strongly convex setting for 66 iterations}
        \label{fig:subfig2}
    \end{subfigure}
    
    \caption{Performance of $\augconex$ vs ConEx in terms of optimality and feasibility gaps on 10 i.i.d instances of sparse QCQP using projection in Algorithm \ref{alg:implicit}.}
    \label{fig:proj}
\end{figure}
We should mention that although $\augconex$ is a two-loop algorithm, the corresponding operator for computing implicit step in \eqref{eq:conceptual-x-s-update} converges to its fixed point in a few numbers of iterations. Specifically, the inner loop (Algorithm \ref{alg:implicit}
) reaches a fixed point in no more than four iterations. Figure \ref{fig:proj-iter} gives a histogram of the number of proximal operations in the inner loop in both variants of experiments in the convex setting. 
\begin{figure}[H]
    \centering
    \begin{subfigure}{.45\textwidth}
        \centering
        \includegraphics[width=.9\linewidth]{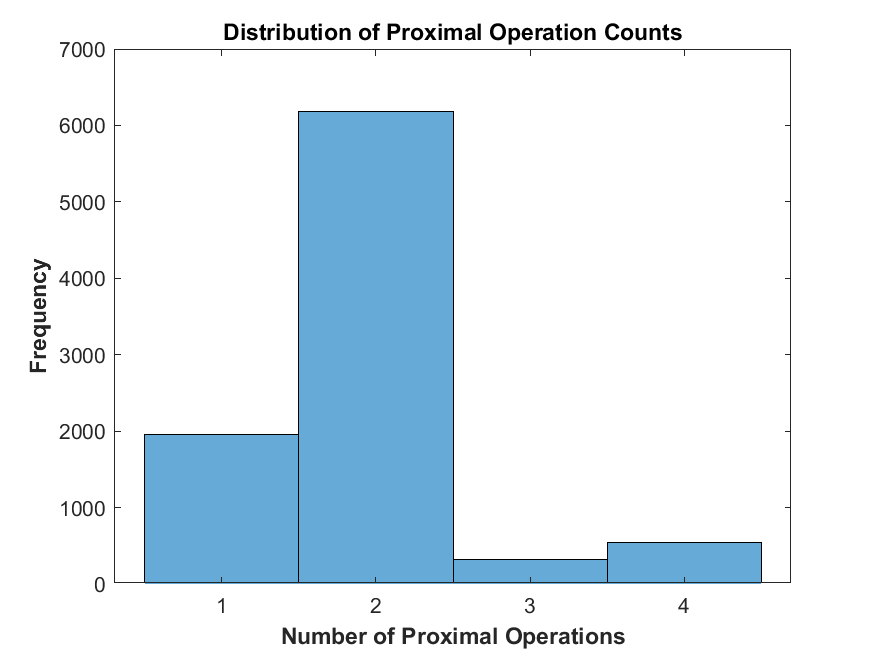}  
        \caption{First variant (proximal operator in Algorithm \ref{alg:implicit})}
        \label{fig:subfig1-hist}
    \end{subfigure}
    \hfill
    \begin{subfigure}{.45\textwidth}
        \centering
        \includegraphics[width=.9\linewidth]{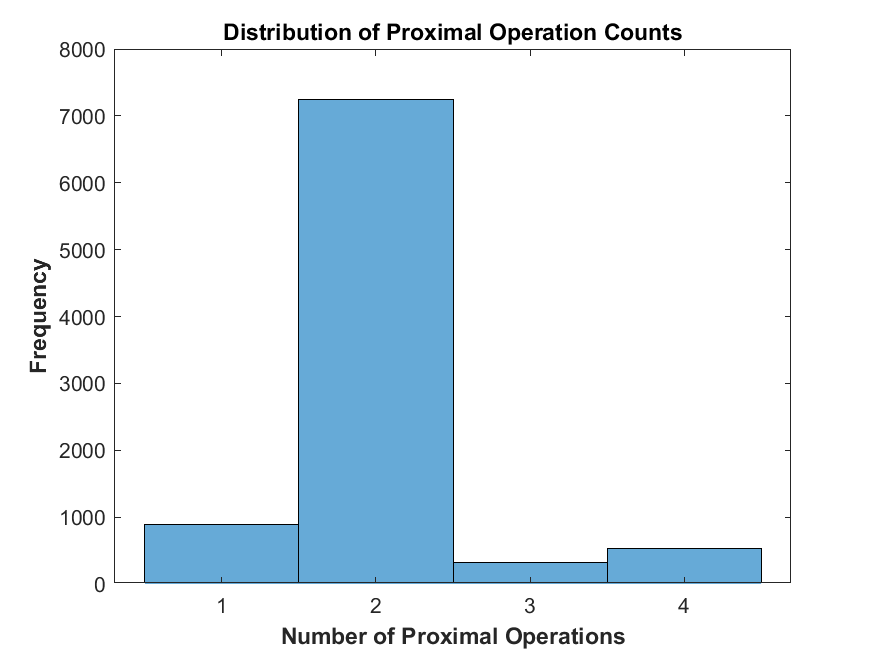}  
        \caption{Second variant (projection in Algorithm \ref{alg:implicit})}
        \label{fig:subfig2-hist}
    \end{subfigure}
    
    \caption{Number of iterations in Algorithm \ref{alg:implicit} with $K=9000$ to converge to the fixed points of \eqref{eq:conceptual-x-s-update} in Convex setting.}
    \label{fig:proj-iter}
\end{figure}

Another interesting measure is the sparsity level of solutions. We compare $\augconex$ and ConEx algorithms on different sparse QCQPs by varying the parameter $\lambda$. The sparsity level is determined by the number of zero elements of solutions produced by each method. From the practical view, we define sparsity level as the number of elements with absolute value less than $ 10^{-10}$. Table \ref{tab:prox} depicts the superiority of the $\augconex$ method since it can handle the sparsity conditions more effectively as the penalty parameter increases.
On the other hand, ConEx fails to control the sparsity given that it returns an averaging solution. It is worth mentioning that although the criterion for counting an element as zero is to have its absolute value less than $10^{-10}$, all the aspects satisfying this condition have exactly zero value. This means that $\augconex$ can return solutions much smaller than our initial criterion for sparsity level. In the second variant of experiments, both algorithms give poor solutions. This failure is somehow predictable since we use a linear approximation of the penalty function in the second variant instead of its exact form in the first variant. 
\begin{table}[H]
  \caption{Sparsity level of the solutions (i.e., number of elements with absolute value $< 10^{-10}$) produced by $\augconex$ and ConEx algorithms for sparse QCQP with $n=100$ in the first variant of the experiments. }
\begin{adjustbox}{width=1\textwidth}
\begin{tabular}{|c|c|c|c|c|c|c|}
\hline
\diagbox{Algorithm}{Penalty parameter} & $\lambda = 20$ & $\lambda = 22$ & $\lambda = 24$ & $\lambda = 26$ & $\lambda = 28$& $\lambda = 30$\\
\hline
& Convex | Strongly convex & Convex | Strongly convex  & Convex | Strongly convex  & Convex | Strongly convex  & Convex | Strongly convex & Convex | Strongly convex  \\
\hline
 $\augconex$& 1 | 21 & 29 | 38 & 75 | 88 & 97 | 97& 100 | 99 &100 | 100\\
\hline
ConEx & 0 | 0 & 0 | 0 & 0 | 0 & 0 | 0 & 0 | 0 &0 | 0\\
\hline
\end{tabular}
\end{adjustbox}
    \label{tab:prox}
\end{table}

Overall, numerical experiments on sparse QCQPs corroborate our findings in terms of convergence rate and show $\augconex$ addresses the issues we have in iterate-averaging methods such as ConEx in \citep{boob2023stochastic} concerning other measures such as sparsity level. 

\section{Conclusion}\label{sec:conclusion}

\appendix
\section{Analysis of Algorithm \ref{alg:x-s-update} }\label{sec:implicit}
This appendix provides the technical results used to obtain Lemma \ref{lem:F-contraction}. First, we need the following proposition. 
\begin{proposition}\label{prop:a-b-relation}
	For vectors $a, b$, we have $\gnorm{[a]_+ - [b]_+}{}{} \le \gnorm{a-b}{}{}$
\end{proposition}
\begin{proof}
	It suffices to prove the following inequality for any two scalars $s_1, s_2$, we have  \[|[s_1]_+-[s_2]_+| \le |s_1-s_2|.\] Indeed, if the above relation holds, applying this elementwise to vectors $a, b$ yields the desired result.
	We divide it into three cases:\\
	{\bf Case 1:} $s_1 \ge 0, s_2 \ge 0$. Then, $|[s_1]_+-[s_2]_+|  = |s_1-s_2|$.\\
	{\bf Case 2:} $s_1 < 0, s_2 < 0$. Then, $|[s_1]_+-[s_2]_+| = 0 \le |s_1-s_2|$.\\
	{\bf Case 3:} Either $s_1 < 0$ or $s_2 < 0$. Without loss of generality, we assume $s_2 < 0$. Then, $|[s_1]_+-[s_2]_+| = s_1 \le |s_1-s_2|$.\\
	In all three cases, we showed the required relation. Hence, we conclude the proof.
\end{proof}
\subsection{Proof of Lemma \ref{lem:F-contraction}}
	Note that 
\[
s_{k+1}=[U_k + g'(\wh{x}_k)^\top x_{k+1}+\chi (x_{k+1})]_-.
\]
Consequently, from the definition of $y_{k+1}$ in \eqref{eq:Denotes}, one can say 
\[
y_{k+1} = [U_k + g'(\wh{x}_k)^\top x_{k+1}+\chi (x_{k+1})]_+.
\]
Replacing this value into $x_{k+1}$ update in Algorithm \ref{alg:augConEx} and by equivalency between \eqref{eq:prox-operaotr1} and \eqref{eq:prox-operaotr2}, we have 
\begin{equation*}
	\begin{aligned}
		x_{k+1}=& \proximal_{\Gamma(x;x_{k+1})/L_k}\paran[\Big]{\wh{x}_k - \tfrac{1}{L_k}\braces[\big]{\mathfrak{F}(\wh{x}_k,\xi_k) + \rho_k g'(\wh{x}_k)[U_k + g'(\wh{x}_k)^\top x_{k+1}+\chi (x_{k+1})]_+}}\\
		=& F(x_{k+1}).
	\end{aligned}
\end{equation*}
From the above relation, it is clear that $x_{k+1}$ is a fixed point of the operator $F(\bx)$. 
Next, let us define $\mathcal{T}(\bx)$ such that 
\[
\mathcal{T}(\bx) = \wh{x}_k - \tfrac{1}{L_k}\braces[\big]{\mathfrak{F}(\wh{x}_k,\xi_k) + \rho_k g'(\wh{x}_k)[U_k + g'(\wh{x}_k)^\top \bx+\chi (\bx)]_+},
\]
Also suppose $\mathcal{T}_1= \mathcal{T}(\bx_1),\mathcal{T}_2= \mathcal{T}(\bx_2)$ and let $u_1 = F(\bx_1), u_2 = F(\bx_2)$. Then, one can construct the following relations
\begin{subequations}
	\begin{equation*}
		\partial\tfrac{\chi_0(u_1)}{L_k }+ \partial\tfrac{\chi(u_1)}{L_k}^\top c(\bx_1) + u_1-\mathcal{T}_1 \ni \mathbf{0},
	\end{equation*}
	\begin{equation*}
		\partial\tfrac{\chi_0(u_2)}{L_k }+ \partial\tfrac{\chi(u_2)}{L_k}^\top c(\bx_2) + u_2-\mathcal{T}_2 \ni \mathbf{0}.
	\end{equation*}
\end{subequations}
Consequently, 
\begin{subequations}
	\begin{equation*}
		\partial_{\tfrac{\chi_0+\chi^\top c(\bx_1)}{L_k}} (u_1) \in \mathcal{T}_1-u_1,
	\end{equation*}
	\begin{equation*}
		\partial_{\tfrac{\chi_0+\chi^\top c(\bx_1)}{L_k}} (u_2) \in \mathcal{T}_2-u_2 - \tfrac{\chi'(u_2)}{L_k}^\top(c(\bx_2)-c(\bx_1)),
	\end{equation*}
\end{subequations}
where $\tfrac{\chi'(u_2)}{L_k}$ is a subgradient in $\partial \tfrac{\chi(u_2)}{L_k}$.  Using the fact that $\Gamma(\cdot;\bx_1)/L_k$ is a convex function, then $\partial_{\tfrac{\chi_0+\chi^\top c(\bx_1)}{L_k}}$ is a monotone operator, we have 
\begin{equation*}
	\inner{\mathcal{T}_1-u_1-\mathcal{T}_2+u_2 +  \tfrac{\chi'(u_2)}{L_k}^\top(c(\bx_2)-c(\bx_1))}{u_1-u_2}\geq 0.
\end{equation*}
Hence one can conclude the following inequality
\[
\inner{\mathcal{T}_1-\mathcal{T}_2}{u_1-u_2} + \inner{\tfrac{\chi'(u_2)}{L_k}^\top(c(\bx_2)-c(\bx_1))}{u_1-u_2}\geq \gnorm{u_1-u_2}{}{2}.
\]
By Cauchy-Schwarz inequality, assumptions in \eqref{eq:M_g_1} and definition of $c(\bx)$ in \eqref{eq:prox-operaotr2}, one can write the following 
\begin{equation}\label{eq:proj_1}
	\gnorm{\mathcal{T}_1-\mathcal{T}_2}{}{} + \tfrac{\rho_k M_\chi(M_g+M_\chi)}{L_k}\gnorm{\bx_1-\bx_2}{}{}\geq \gnorm{u_1-u_2}{}{}.
\end{equation}
Moreover, from the definition of $\mathcal{T}(\bx)$, one can have the following relations for $\gnorm{\mathcal{T}_1-\mathcal{T}_2}{}{}$ 
\begin{equation}\label{eq:proj2}
	\begin{aligned}
		\gnorm{\mathcal{T}_1-\mathcal{T}_2}{}{}& = \big\| \bracket[\big]{\wh{x}_k - \tfrac{1}{L_k}\braces[\big]{\mathfrak{F}(\wh{x}_k,\xi_k) + \rho_k g'(\wh{x}_k)[U_k + g'(\wh{x}_k)^\top \bx_1+\chi (\bx_1)]_+}} \\
		&\qquad - \bracket[\big]{\wh{x}_k - \tfrac{1}{L_k}\braces[\big]{\mathfrak{F}(\wh{x}_k,\xi_k) + \rho_k g'(\wh{x}_k)[U_k + g'(\wh{x}_k)^\top \bx_2+\chi (\bx_2)]_+}} \big\|\\
		&= \tfrac{\rho_k}{L_k} \gnorm{g'(\wh{x}_k)\braces{[U_k + g'(\wh{x}_k)^T\bx_2 + \chi(\bx_2)]_+ - [U_k + g'(\wh{x}_k)^T\bx_1+\chi(\bx_1)]_+} }{}{}\\
		&\le \tfrac{\rho_kM_g}{L_k} \gnorm{[U_k + g'(\wh{x}_k)^T\bx_2 + \chi(\bx_2)]_+ - [U_k + g'(\wh{x}_k)^T\bx_1+\chi(\bx_1)]_+}{}{}\\
		&\le \tfrac{\rho_kM_g}{L_k} \gnorm{g'(\wh{x}_k)^T(\bx_2-\bx_1) + \chi(\bx_2)-\chi(\bx_1)}{}{}\\
		&\le \tfrac{\rho_kM_g (M_g + M_\chi)}{L_k} \gnorm{\bx_1-\bx_2}{}{},
	\end{aligned}
\end{equation}
where  the first and third inequalities follow by Assumptions in \eqref{eq:M_g_1}, and second inequality follows by Proposition \ref{prop:a-b-relation}. Therefore, using \eqref{eq:proj_1} and \eqref{eq:proj2}, we have
\[
\gnorm{u_1-u_2}{}{}\leq \tfrac{\rho_k(M_g+M_\chi)^2}{L_k}\gnorm{\bx_1-\bx_2}{}{}.
\]
Taking $L_k\geq 2\rho_k(M_g + M_{\chi})^2$ becomes $F$ a contraction, we note that $x_{k+1}$ is unique fixed point.
Hence, we conclude the proof. $ \openbox $
\subsection{Proof of Theorem \ref{thm:linear-convergence}}
	Note that 
\begin{align*}
	\gnorm{w_{t+1} - x_{k+1}}{}{} &= \gnorm{F(w_t) - x_{k+1}}{}{}\\
	&= \gnorm{F(w_t) - F(x_{k+1})}{}{}\\
	&\le \tfrac{1}{2} \gnorm{w_t - x_{k+1}}{}{},
\end{align*}
where first equality follows by definition of $w_{t+1}$, second equality uses the fact that $x_{k+1}$ is a fixed point of $F$, and final inequality uses contraction property from Lemma \ref{lem:F-contraction}. Exapnding the above recursion, we obtain
\[ \gnorm{w_{T} - x_{k+1}}{}{} \le 2^{-T}\gnorm{w_0 - x_{k+1}}{}{} =  2^{-T}\gnorm{\wh{x}_k - x_{k+1}}{}{},\]
where the last relation uses $w_0 = \wh{x}_k$ as stated in Algorithm \ref{alg:x-s-update}. It is also clear each iteration of Algorithm \ref{alg:x-s-update} requires computing $F(w_t)$ only. Hence, we conclude the proof. $ \openbox $ 
\section{Analytical solution to the proximal operator of an $\ell_1$-norm function under an $\ell_2$-ball constraint}\label{sec:kkt_l1l2}
We provide an explicit solution for Algorithm \ref{alg:implicit} of $\augconex$ method in the first variant of experimented in Section \ref{sec:numerical} where we consider $\ell_1$-norm as a prox-friendly function and project it onto a $\ell_2$-ball with radius $D_x$.
\begin{proposition}\label{prop:sparse}
    Assume the following problem 
    \begin{equation}\label{eq:cons-sparse}
        \min_{\gnorm{x}{}{}\leq D_x}\{\lambda\gnorm{x}{1}{} + \tfrac{1}{2}\gnorm{x-\ddot{x}}{}{2}\},
    \end{equation}
   where $\lambda\geq 0$ and $x,\ddot{x}\in \mathbb{R}^n$. Then we have the corresponding solution to the above problem. 
    \begin{equation}\label{eq:sparse-sol}
        x^{*,i} = x_0^{u,i}\min(1,\tfrac{D_x}{\gnorm{x_0^u}{}{}})\quad \forall i\in [n],
    \end{equation}
    where $ x_0^{u}$ is the solution to the corresponding unconstrained problem in \eqref{eq:cons-sparse} without any set projection assumption and has the following expression
    \begin{equation}\label{eq:x_uncon}
        x_0^{u,i} =\sign (\ddot{x}^i)\max(|\ddot{x}^i|-\lambda,0)\quad \forall i\in [n].
    \end{equation}
    \begin{proof}
        First, considering the equivalent constraint $\gnorm{x}{}{2}\leq D_x^2$ instead of $\gnorm{x}{}{}\leq D_x$, let us write down the KKT conditions corresponding to the Lagrangian relaxation of problem \eqref{eq:cons-sparse}, namely $\mathcal{L}(x,\dot{\mu})$. Note that since Slater's condition holds, we are sure of the existence of the optimal solution. 
         \begin{equation}\label{eq:KKT}
  \text{KKT conditions} =
    \begin{cases}
      \lambda \sign{x^i} + (x^i-\ddot{x}^i) + 2\dot{\mu}x^i = 0 & \text{ Stationarity, } \forall i\in [n] \\
      \dot{\mu} (\gnorm{x}{}{2}-D_x^2) = 0& \text{ Complementary slackness} \\
     \gnorm{x}{}{2}\leq D_x^2  & \text{ Primal Feasibility }\\
     \dot{\mu}\geq 0 & \text{ Dual Feasibility },\\
    \end{cases}       
\end{equation}
        Solving the Stationarity condition of \eqref{eq:KKT} gives us the following solution with respect to $\dot{\mu}$ for all $i\in [n]$
        \begin{equation} \label{eq:uncon}
         x^{u,i}(\dot{\mu}) =
    \begin{cases}
      \tfrac{\ddot{x}^i - \lambda}{2\mu+1} & \text{if } \ddot{x}^i >\lambda \\
      0& \text{if } |\ddot{x}^i|\leq \lambda\\
     \tfrac{\ddot{x}^i + \lambda}{2\mu+1}  & \text{if } \ddot{x}^i <-\lambda ,
    \end{cases}       
\end{equation}
        Next, let us break \eqref{eq:uncon} into two cases. First consider $\mu = 0$. Then, $x^u$ resolves to $x^{u,i}  =x_0^{u,i}$ for all $i\in [n]$. If $x^u$ satisfies primal feasibility then it is the optimal solution of \eqref{eq:cons-sparse}. Otherwise, we have the following solution form for \eqref{eq:cons-sparse}. 
        \begin{equation}\label{eq:x_con-mu-0}
        x^{*,i} = \tfrac{\sign (\ddot{x}^i)\max(|\ddot{x}^i|-\lambda,0)}{\sqrt{\tsum_{i=1}^n \sign (\ddot{x}^i)\max(|\ddot{x}^i|-\lambda,0)}}D_x, \quad \forall i\in [n].
        \end{equation}
        Now, lconsider $\dot{\mu}>0$. From Complementary slackness of \eqref{eq:KKT}, one can say $\gnorm{x^*}{}{} = D_x$. Letting $\dot{\mu} = \tfrac{\sqrt{\tsum_{i=1}^n \sign (\ddot{x}^i)\max(|\ddot{x}^i|-\lambda,0)}}{2D_x} - \tfrac{1}{2} $, we obtain optimal solution $x^*$ as expressed in \eqref{eq:x_con-mu-0}. Note that we know $\sqrt{\tsum_{i=1}^n \sign (\ddot{x}^i)\max(|\ddot{x}^i|-\lambda,0)}>D_x$ since $x^{u}(\dot{\mu})$ is a infeasible solution for $\dot{\mu}>0$. 
    \end{proof}
\end{proposition}
\bibliographystyle{abbrvnat}
\bibliography{submit_arxiv}

\end{document}